\documentclass{amsart}
\usepackage{amssymb,latexsym}
\usepackage{amscd,amsthm}

\usepackage[all]{xy}
\usepackage{tikz}

\usepackage{tikz-cd}

\newtheorem{theorem}{Theorem}[section]

\newtheorem{lemma}[theorem]{Lemma}
\newtheorem{proposition}[theorem]{Proposition}
\newtheorem{corollary}[theorem]{Corollary}

\theoremstyle{definition}
\newtheorem{definition}[theorem]{Definition}

\newtheorem{remark}[theorem]{Remark}
\newtheorem*{notation}{Notation}
\newtheorem{example}[theorem]{Example}

\DeclareMathOperator{\Ext}{Ext}
\DeclareMathOperator{\Hom}{Hom}

\DeclareMathOperator{\Ker}{Ker}

 \newcommand{\Oplus}{\ensuremath{\vcenter{\hbox{\scalebox{1.4}{$\oplus$}}}}} 				
 
\newcommand{\cat}[1]{\mathcal{#1}}           

\newcommand{\tensor}{\otimes}

\newcommand{\class}[1]{\mathcal{#1}}   

\newcommand{\Z}{\mathbb{Z}}
\newcommand{\Q}{\mathbb{Q/Z}}

\newcommand{\ch}{\textnormal{Ch}(R)}

\newcommand{\rmod}{R\textnormal{-Mod}}

\newcommand{\tilclass}[1]{\widetilde{\class{#1}}}
\newcommand{\dgclass}[1]{dg\widetilde{\class{#1}}}
\newcommand{\dwclass}[1]{dw\widetilde{\class{#1}}}

\newcommand{\rightperp}[1]{#1^{\perp}}
\newcommand{\leftperp}[1]{{}^\perp #1}

\begin{document}

\title{Acyclic complexes of FP-injective modules over Ding-Chen rings}

\author{James Gillespie}
\address{J.G. \ Ramapo College of New Jersey \\
         School of Theoretical and Applied Science \\
         505 Ramapo Valley Road \\
         Mahwah, NJ 07430\\ U.S.A.}
\email[Jim Gillespie]{jgillesp@ramapo.edu}
\urladdr{http://pages.ramapo.edu/~jgillesp/}

\date{\today}



\begin{abstract}
 We present a new method for combining two cotorsion pairs to obtain  an abelian model structure  and we apply it to  construct and study a new  model structure  on left $R$-modules over a left coherent ring $R$. Its class of fibrant objects is generated by the weakly Ding injective $R$-modules, a class of modules recently studied by Iacob.  We give several characterizations of the fibrant modules, one being that they are the cycle modules of certain acyclic complexes of FP-injective (i.e., absolutely pure)  $R$-modules. In the case that  $R$ is a Ding-Chen ring, we show that they are precisely the modules appearing as cycles of acyclic complexes of FP-injectives. This leads to  a new descrption of $\textnormal{Stmod}(R)$, the  stable module category of a Ding-Chen ring $R$, by way of modules we call Gorenstein FP-pro-injective. These are modules that appear as a cycle module of a totally acyclic complex of FP-projective-injective modules. As  a completely separate application of  the new model category method, we show that all complete cotorsion pairs, even non-hereditary ones, lift to   abelian models for the derived category of a ring.   
\end{abstract}

\maketitle

\section{Introduction}\label{section-intro}

Let $R$ be a (two-sided) coherent ring. A (say left) $R$-module $A$ is called  \emph{FP-injective}  if  $\Ext^1_R(M,A) = 0$ for all finitely presented modules $M$.  This condition on a module $A$  is satisfied   if and only if  any short exact sequence  with $A$ as its first term must be a  pure exact sequence. For this reason FP-injective modules are  also  called   \emph{absolutely pure}  modules, and are often viewed as a dual notion to that of a flat module.  Indeed  an $R$-module $F$ is flat if and only if any short exact sequence  whose last term is $F$ is a pure exact sequence. The coherence of the ring $R$ ensures that FP-injectives satisfy other dual analogs to useful properties of flat modules. Moreover,     the functor $\Hom_{\Z}(-,\Q)$ provides a perfect duality between FP-injective left (resp. right) $R$-modules and flat right (resp. left) $R$-modules.

In relative homological algebra, a left $R$-module $M$ is called \emph{Ding projective} if it equals the 0-cycles, $M = Z_0X$, of some exact (i.e. acyclic) chain complex $X$ of projective  modules for which $\Hom_R(X,F)$ is also exact for any flat module $F$. Equivalently, $A \otimes X$ is exact for any injective, or even just for any FP-injective   right $R$-module $A$. This  means that the Ding projective modules coincide with the \emph{projectively coresolved Gorenstein flat modules}, or \emph{PGF modules}, that appeared in~\cite{saroch-stovicek-G-flat}.
There is also the dual notion of a Ding injective module: $M$ is called \emph{Ding injective} if $M = Z_0X$ for some exact chain complex $X$ of injective modules for which $\Hom_R(A,X)$ is exact for any FP-injective  module $A$. 

These classes of Ding projective and Ding injective modules are the cofibrant (resp. fibrant) objects of two different abelian model structures on $\rmod$, the category  of (left)  $R$-modules. We recall  from~\cite{hovey, gillespie-book}  that an abelian model structure $\mathfrak{M}$, on $\rmod$, is equivalent to a thick class $\class{W}$ of `trivial' modules linking together two complete cotorsion pairs, $(\class{Q}, \class{W}\cap\class{R})$ and $(\class{Q}\cap\class{W},\class{R})$. We denote this information by way of a triple $\mathfrak{M} = (\class{Q},\class{W},\class{R})$ and call $\class{Q}$ (resp. $\class{R}$) the class of cofibrant (resp. fibrant) objects.

Letting $\class{DP}$ denote the class of all Ding projective $R$-modules, there is an abelian model structure on $\rmod$, represented by the triple  $$\mathfrak{M}_{proj} = (\class{DP}, \class{V},  All).$$ 
Every module is fibrant in this model structure and the class of trivial modules is $\class{V} := \rightperp{\class{DP}}$, where this  denotes the right orthogonal of $\class{DP}$ with respect to $\Ext^1_R(-,-)$. 
The dual is also true: Letting  $\class{DI}$ denote the class of all Ding injective $R$-modules, we have  an abelian model structure  $$\mathfrak{M}_{inj} = (All, \class{W}, \class{DI}),$$
 where $\class{W} := \leftperp{\class{DI}}$ is the left Ext-orthogonal of $\class{DI}$. 
 
Now let $\class{GF}$ denote the class of all Gorenstein flat modules. Such modules $M$ equal the 0-cycle module, $M = Z_0X$, of some \emph{F-totally acyclic} complex of flat modules, meaning $I \otimes X$ is exact for any injective right $R$-module $I$.  
Then there is another abelian model structure on $\rmod$,  represented by the triple  $$\mathfrak{M}_{flat} = (\class{GF}, \class{V},  \class{C}).$$ 
Note in particular that its class $\class{V}$ of trivial objects is the same as those in $\mathfrak{M}_{proj}$.  But now the  class of fibrant objects is the class $\class{C}$   of   all cotorsion modules, that is, the class of all modules $C$ having the property that $\Ext^1_R(F,C) = 0$ for all flat modules $F$.

$\mathfrak{M}_{proj}$ and  $\mathfrak{M}_{inj}$ are dual to one another, and we  show in Theorem~\ref{theorem-weakly-Ding-injectives}   that $\mathfrak{M}_{flat}$ also has a dual. 
To describe it, let $\class{FI}$ denote  the class of all FP-injective modules.     
We have a complete cotorsion pair $(\class{FP}, \class{FI})$, and the modules  in $\class{FP} := \leftperp{\class{FP}}$ are called \emph{FP-projective}. Note how the definition of the class $\class{FP}$ of all FP-projective modules is dual to the definition of the class $\class{C}$ of all  cotorsion modules.  
Now, Theorem~\ref{theorem-weakly-Ding-injectives}  asserts that we have an 
abelian model structure on $\rmod$,  represented by the triple 
$$\mathfrak{M}_{fp} = (\class{FP}, \class{W}, \rightperp{(\leftperp{\textnormal{w}\class{DI}})}).$$
In particular, it has  the same class of trivial objects, $\class{W} := \leftperp{\class{DI}}$, as in $\mathfrak{M}_{inj}$.
The class of fibrant modules is generated by the class 
$\textnormal{w}\class{DI}$, of \emph{weakly Ding injective} modules, 
a class of modules going back to~\cite{Goren-FP-inj} and much studied by Iacob in the recent works~\cite{iacob-weakly-Ding-1, iacob-weakly-Ding-2, iacob-weakly-Ding-3}. We hesitate to assign a  specific notation for the class  $\rightperp{(\leftperp{\textnormal{w}\class{DI}})}$, and to choose a name for these modules. However,  it is tempting to name them \emph{Ding FP-injective} modules, for  they have many properties that are dual to the Gorenstein flat modules, and we note that the Gorenstein flat  modules  are the same thing as the notion of a \emph{Ding flat} module. I.e., it is equivalent to let the modules $I$ in the  above notion of F-totally acyclic  range through all FP-injective modules, not just all injective modules. This is true for any ring, not just coherent ones;  see~\cite[Prop.~7.2]{gillespie-ding-modules}.

The most interesting case  of the model structure  $\mathfrak{M}_{fp}$ of  Theorem~\ref{theorem-weakly-Ding-injectives} is when $R$ is taken to be a Ding-Chen ring, in the sense of~\cite{ding and chen 93} and~\cite{gillespie-Ding-Chen rings}. In this case, the two classes of trivial objects identify, $$\class{W} = \leftperp{\class{DI}} = \rightperp{\class{DP}} = \class{V},$$ and they   coincide with the class of all modules of finite flat dimension, or equivalently, of finite FP-injective dimension. We now summarize the results of the paper for the case of a Ding-Chen ring $R$, using the following notation and terminology:  Again,   $(\class{FP}, \class{FI})$ denotes the complete cotorsion pair where  $\class{FI}$ is  the class of all  FP-injective modules, and   $\class{FP}$ is the class of all FP-projective modules.  We call a module $M$ in the core of this cotorsion pair, i.e.,  $M \in \class{FP}\cap\class{FI}$, an \emph{FP-pro-injective} module. Then by a  \emph{totally acyclic complex of FP-pro-injective} modules we mean an exact complex $X$ with each $X_n$ FP-pro-injective and such that both $\Hom_R(M, X)$ and $\Hom_R(X, M)$ are also exact complexes whenever $M$ is FP-pro-injective. Finally, by a \emph{Gorenstein FP-pro-injective} module we mean one that is equal to the 0-cycle module of some totally acyclic complex of FP-pro-injective modules. 
 
\begin{theorem}\label{them-intro}
Let $R$ be a  Ding-Chen ring. 
Then each of the following hold.  
\begin{itemize}
\item (Proposition~\ref{prop-exact-complexes-FP-injectives}.) $\rightperp{(\leftperp{\textnormal{w}\class{DI}})} = \class{Z}$, where $\class{Z}$ denotes the class of all cycles of acyclic complexes of FP-injective modules. Moreover,  
any acyclic  chain complex $X$ of FP-injectives automatically has the property that $\Hom_R(M,X)$ is exact for any FP-pro-injective module $M$.
\item (Proposition~\ref{prop-exact-complexes-FP-injectives}.)   $\class{FP}\cap\class{Z} = \class{GFP}$, where $\class{GFP}$ denotes the class of all Gorenstein FP-pro-injective modules.
Moreover,  any acyclic  chain complex $X$ of FP-pro-injectives is automatically  totally acyclic. 
\item (Theorem~\ref{theorem-exact-FP-injectives}.) The triple $\mathfrak{M}_{fp} = (\class{FP}, \class{W}, \class{Z})$ represents a cofibrantly generated  abelian model structure on $R\textnormal{-Mod}$.
\item (Theorem~\ref{theorem-exact-FP-injectives}.)  The stable module category of $R$, denoted $\textnormal{Stmod}(R)$, coincides with the homotopy category of $\mathfrak{M}_{fp}$  , yielding a triangle equivalence 
 $$\textnormal{Stmod}(R) = \textnormal{Ho}(\mathfrak{M}_{fp}) \simeq \textnormal{St}(\class{GFP}).$$ 
Here,  $\textnormal{St}(\class{GFP})$ denotes the stable category of  the Frobenius category of all Gorenstein FP-pro-injective modules. 
\item  (Theorem~\ref{theorem-exact-FP-injectives}.) We have the following equivalent statements, further characterizing the class of fibrant objects: 
\begin{enumerate}
\item $M \in \class{Z}  = \rightperp{(\leftperp{\textnormal{w}\class{DI}})} = \rightperp{(\leftperp{(\class{FI}\cup\class{DI})})}$. That is,  $M=Z_{0}X$ for some acyclic  complex $X$ of FP-injective modules which necessarily has the property that $\Hom_R(M, X)$ is exact for any FP-pro-injective $M$.
\item There is a short exact sequence $0 \xrightarrow{} D \xrightarrow{} A \xrightarrow{} M \xrightarrow{}0$ with $A \in \class{FI}$ and $D\in\class{DI}$.
\item $\Ext^i_R(W,M)=0$ for all FP-pro-injectives $W$ and all $i\geq1$.
\item There is a short exact sequence $0 \xrightarrow{} M \xrightarrow{} D \xrightarrow{} A \xrightarrow{}0$ with $A \in \class{FI}$ and $D\in\class{DI}$ with the property that is remains exact upon applying $\Hom_R(M,-)$ for any FP-pro-injective module $M$.
\end{enumerate}
\item (Theorem~\ref{theorem-duality}.)
$(\class{GF},\class{Z})$ is a complete duality pair,  in the sense of~\cite[Def.~7]{gillespie-iacob-duality-pairs}. In particular, $(\class{GF},\class{Z})$ and $(\class{Z}, \class{GF})$ are both coproduct and product closed duality pairs in the sense of~\cite[Def.~2.1]{holm-jorgensen-duality}. 
 \end{itemize}
\end{theorem}

Our results are in line with Dalezios' work~\cite{dalezios-FP-injective} on complexes of  FP-pro-injectives, and provides us with the duals  (in the Ding-Chen case) of the  flat-cotorsion theory from~\cite{cet-totally-flat-cot-theory}. By combining Theorem~\ref{them-intro}  with other known results from the literature,   $\textnormal{Stmod}(R)$,  the stable module category of a Ding-Chen ring $R$, now  has four representations:  It is equivalent to the category of all  modules appearing as a cycle  of some acyclic complex of projective (resp.  injective, resp.  flat-cotorsion, resp.  FP-pro-injective) $R$-modules. Moreover, in each case, any such complex is totally acyclic. See~\cite[Theorems~4.7/4.10]{gillespie-Ding-Chen rings}, \cite[Theorem~1.2]{gillespie-ding-modules}, and~\cite[Theorems~4.4/5.2]{cet-totally-flat-cot-theory}.

Theorems~\ref{theorem-cots-models} and~\ref{theorem-fibrant-objects} are the theoretical underpinning of the aforementioned results.  They provide a  method for constructing a new abelian model structure,  from a given one having every object either cofibrant or fibrant.  These results   are interesting on their own, and in Section~\ref{section-non-hereditary}
we give a second application of them. 
To explain it briefly, 
again  let   $\rmod$  be the category of modules over a ring, and now let $\ch$ denote  the corresponding category of  chain complexes.  In~\cite{gillespie} we described a recipe for lifting cotorsion pairs $(\class{A},\class{B})$  in $\rmod$ to cotorsion pairs yielding potential model structures on $\ch$. This procedure required that the cotorsion pair be \emph{hereditary} in the sense that $\Ext^i_R(A,B) = 0$ for all $A\in\class{A}$, $B\in\class{B}$, and $i\geq 1$. 
A question that I've been asked more than once  is whether or not there are examples of non-hereditary abelian model structures. Certainly, once one has a non-hereditary complete cotorsion pair, and these certainly exist, then we trivially obtain a non-hereditary abelian model structure. For instance, if $(\class{C},\class{F})$ is any complete cotorsion pair, then we trivially obtain an abelian model structure $(\class{C}, \class{W}, \class{F})$, by taking $\class{W}$ to be the class of all objects. But such   examples are unsatisfying  as their corresponding homotopy category is zero. We would like to see nontrivial examples. 
 Theorem~\ref{theorem-non-hereditary-derived-models}, which follows from Theorems~\ref{theorem-cots-models} and~\ref{theorem-fibrant-objects}, shows  how the method from~\cite{gillespie} extends to non-hereditary cotorsion pairs. 

We now summarize the layout of the paper.  After the preliminary Section~\ref{section-prelims},  Section~\ref{section-merging-models} is dedicated to proving the new model category results,  Theorems~\ref{theorem-cots-models} and~\ref{theorem-fibrant-objects}. We give a general presentation that works for Quillen exact categories, and which also have duals. Section~\ref{section-non-hereditary} gives the quick application to the construction of non-hereditary model structures for the derived category of a ring. In Section~\ref{subsec-tech-acyclic} we return to the general setting of Section~\ref{section-merging-models}  in order to prove some results about the homotopy category associated to the new model structures. The main result here is Theorem~\ref{theorem-stable-CF} which provides conditions ensuring that the homotopy category can be described in terms of totally acyclic complexes. In Section~\ref{section-weakly-B} we generalize some of Iacob's results from~\cite{iacob-weakly-Ding-3}, and then relate them  to   the existence of certain abelian model structures in  Gorensteing homological algebra. The model structure $\mathfrak{M}_{fp} = (\class{FP}, \class{W}, \rightperp{(\leftperp{\textnormal{w}\class{DI}})})$ on modules over a left coherent ring is a special case, and it is studied in Section~\ref{section-weakly-Ding-model}. In Section~\ref{Section-Ding-Chen} we specialize to Ding-Chen rings and prove the remainder of Theorem~\ref{them-intro}.


\section{Preliminaries}\label{section-prelims}

We summarize here the main terminology and  references  from relative homological algebra that we will use. Most fundamental to the paper are cotorsion pairs, covering classes,  abelian model structures,   and Ding-Chen rings.

\subsection{Cotorsion pairs and abelian model categories}
Let $\cat{A}$ be an abelian category such as $\rmod$, the category  of (left) modules over a ring $R$.  By definition, a pair of classes $(\class{X},\class{Y})$ in $\cat{A}$ is called a \emph{cotorsion pair} if $\class{Y} = \rightperp{\class{X}}$ and $\class{X} = \leftperp{\class{Y}}$. Here, given a class of objects $\class{C}$ in $\cat{A}$, the right orthogonal  $\rightperp{\class{C}}$ is defined to be the class of all objects $X$ such that $\Ext^1_{\cat{A}}(C,X) = 0$ for all $C \in \class{C}$. Similarly, we define the left orthogonal $\leftperp{\class{C}}$. We call the cotorsion pair \emph{hereditary} if $\Ext^i_{\cat{A}}(X,Y) = 0$ for all $X \in \class{X}$, $Y \in \class{Y}$, and $i \geq 1$. The cotorsion pair is \emph{complete} if it has enough injectives and enough projectives. This means that for each $A \in \cat{A}$ there exist short exact sequences $0 \xrightarrow{} A \xrightarrow{} Y \xrightarrow{} X \xrightarrow{} 0$ and $0 \xrightarrow{} Y' \xrightarrow{} X' \xrightarrow{} A \xrightarrow{} 0$ with $X,X' \in \class{X}$ and $Y,Y' \in \class{Y}$.
Standard references for cotorsion pairs of $R$-modules  include the books~\cite{enochs-jenda-book} and~\cite{trlifaj-book}. 

We will use the correspondence between complete cotorsion pairs and  abelian model categories shown in~\cite{hovey}. These  ideas naturally extend to the more general setting of Quillen exact categories, and we will use the book \cite{gillespie-book} to reference standard facts about cotorsion pairs and abelian model structures in this setting. However,  Quillen exact categories only play a role in Sections~\ref{section-merging-models} and~\ref{subsec-tech-acyclic}, and our use of exact categories  is really just for convenience and to have the main results for future reference. The reader so inclined can certainly read Sections~\ref{section-merging-models} and~\ref{subsec-tech-acyclic}  with the category of $R$-modules (or chain complexes) in mind. Afterall, the  applications in the current paper are only to $R$-modules and chain complexes of $R$-modules. 
 
 A  class of objects $\class{W}$ in an abelian, or exact, category $\cat{A}$ is said to be \emph{thick} if it is closed under direct summands and satisfies that whenever 2 out of 3 terms in a short exact sequence are in $\class{W}$, then so is the third. 
Then as mentioned in the Introduction,   an abelian model structure  is equivalent to a thick class $\class{W}$ linked together  with two complete cotorsion pairs, $(\class{Q}, \class{W}\cap\class{R})$ and $(\class{Q}\cap\class{W},\class{R})$. We denote this information by way of a triple $\mathfrak{M} = (\class{Q},\class{W},\class{R})$ which we call a \emph{Hovey triple}.  Then $\class{Q}$ (resp. $\class{R}$) is the class of cofibrant (resp. fibrant) objects, and $\class{W}$  is the class of trivial objects. It is worth mentioning that we need not  distinguish  between a Hovey triple and an abelian model structure when the underlying additive category $\cat{A}$ is weakly idempotent compete. This is a very mild condition; it just means that every split monorphism in $\cat{A}$ possesses a cokernel in $\cat{A}$, or equivalently, every split epimorphism possesses a kernel.
 
For  convenience, we recall what we mean by an injective  cotorsion pair. Let  $(\cat{A},\class{E})$ be a Quillen exact category, so  $\class{E}$  denotes the distinguished class of short exact sequences. Assuming $(\cat{A},\class{E})$ has  enough injectives, we say that a complete  cotorsion pair   $(\class{W}, \class{R})$ in $(\cat{A},\class{E})$ is  an  \emph{injective cotorsion pair}  if  $\class{W}$ is  thick and $\class{W}\cap\class{R}$ coincides with the class of all injective objects.  Assuming $\cat{A}$ is weakly idempotent complete,  an injective cotorsion pair  is equivalent to an abelian model structure on $(\cat{A},\class{E})$ in which every object is cofibrant. The dual notion, defined  in the case that $(\cat{A},\class{E})$ has enough projectives, is that of a  \emph{projective cotorsion pair}.

\subsection{Covers and envelopes and duality pairs}
The book  \cite{enochs-jenda-book} is our standard reference for topics in  relative and Gorenstein  homological algebra. In particular, the notions of (pre)covers and (pre)envelopes of $R$-modules are discussed in \cite{enochs-jenda-book}. Occasionally we will utilize the relationship between these notions and that of a  duality pair of $R$-modules. For this we will give references to~\cite{holm-jorgensen-precovers-purity}, \cite{holm-jorgensen-duality},   \cite{gillespie-duality-pairs},  and~\cite{gillespie-iacob-duality-pairs}.

\subsection{Hom-acyclicity terminology}\label{subsec-Hom-acyclicity}
Let $X$ be a chain complex of $R$-modules, or more generally, a chain complex of objects from some exact category $(\cat{A},\class{E})$. Let $\class{C}$ be   some class of objects of $\cat{A}$.  
We will say that  $X$   is \emph{$\Hom_{\cat{A}}(\class{C},-)$-acyclic} if $\Hom_{\cat{A}}(C,X)$ is an exact complex of abelian groups for all $C \in \class{C}$. Moreover, if $X$ itself is also $\class{E}$-exact, we will say that $X$ is an \emph{exact $\Hom_{\cat{A}}(\class{C},-)$-acyclic} complex. If the similar statement is true in the contravariant variable,  we will say that $X$ is \emph{(exact) $\Hom_{\cat{A}}(-, \class{C})$-acyclic}. We use similar terminology for short exact sequences: We speak of \emph{$\Hom_{\cat{A}}(\class{C},-)$-exact} short exact sequences and  \emph{$\Hom_{\cat{A}}(- , \class{C})$-exact} short exact sequences.

\subsection{Ding-Chen rings}\label{subsection-Ding-Chen}
Recall that a two-sided Noetherian ring $R$ is \emph{Iwanaga-Gorenstein} if  $\textnormal{id}(R_R)$ and $\textnormal{id}({}_RR)$, that is, its injective dimensions when considered as  a left and right module over itself,  are both  finite.  In this case, we have $\textnormal{id}(R_R) = \textnormal{id}({}_RR)$.  Ding and Chen proved a generalization of this to coherent rings: They showed  in~\cite[Corollary 3.18]{ding and chen 93} that if $R$ is a two-sided coherent ring with finite self FP-injective dimensions,   $FP\textnormal{-id}(R_R)$ and $FP\textnormal{-id}({}_RR)$,  then we must have  $FP\textnormal{-id}(R_R)=FP\textnormal{-id}({}_RR)$. Such a ring $R$ is called a \emph{Ding and Chen ring}, and if this common dimension is $d <\infty$, we say that $R$ is a Ding and Chen ring of \emph{dimension} $d$, or a  \emph{$d$-FC ring}.
It is also shown in~\cite{ding and chen 93} that the following statements, concerning the flat and  FP-injective dimensions of an $R$-module $M$, are equivalent: 
$$\bullet \,  \textnormal{fd}(M)  < \infty \ \ \ \ \  \bullet \,  \textnormal{fd}(M)  \leq d    \ \ \ \ \    \bullet \,   \textnormal{FP-id}(M) < \infty  \ \ \ \ \   \bullet \,   \textnormal{FP-id}(M) \leq d.$$
Letting $\class{W}$ denote the class of all such modules $M$, we define its associated \emph{stable module category}  to be 
 $$\textnormal{Stmod}(R) := \rmod/\class{W},$$
 where the right-hand side denotes the triangulated localization 
 of $\rmod$,  the category  $R$-modules,  with respect to $\class{W}$.  The idea of a triangulated localization is made precise in~\cite[\S6.7]{gillespie-book}, but briefly, it  means each of the following hold: (i) $\textnormal{Stmod}(R)$ is a triangulated category, (ii) there is a canonical additive functor $\gamma \colon \rmod \xrightarrow{} \textnormal{Stmod}(R)$ that functorially assigns  short exact sequences to  exact triangles, (iii) $\gamma(W) = 0$ for all $W\in\class{W}$, and, (vi) $\gamma$ is universally initial with respect to the properties (i)--(iii).   
As described in the Introduction,
$\textnormal{Stmod}(R)$ may be realized as the homotopy category of the three abelian model structures, $\mathfrak{M}_{proj}$, $\mathfrak{M}_{inj}$, and $\mathfrak{M}_{flat}$. See~\cite[Theorems~4.7/4.10]{gillespie-Ding-Chen rings}, \cite[Theorem~1.2]{gillespie-ding-modules}, and~\cite[Theorems~4.4/5.2]{cet-totally-flat-cot-theory} for details.



\section{Merging a cotorsion pair with an injective model structure}\label{section-merging-models}

In this section we show how to join an injective (resp.  projective) model structure together with a second complete cotorsion pair. It produces a new model structure having the same trivial objects but whose cofibrant (resp.  fibrant) objects come from the second cotorsion pair.  
The idea works in great generality, so throughout this section we let $(\class{A},\class{E})$ be any  exact category. That is, $\cat{A}$ is an additive category with an exact structure $\class{E}$ in the sense of Quillen~\cite{quillen-algebraic K-theory}.


\begin{lemma}\label{lemma-cots}
Let  $(\class{X},\class{Y})$ and $(\class{X}',\class{Y}')$ be   cotorsion pairs in $(\class{A},\class{E})$.  Then each of these statements hold: 
\begin{enumerate}
\item $(\class{X}\cap \class{X}', \rightperp{(\leftperp{(\class{Y}\cup\class{Y}')})})$ is the cotorsion pair generated by $\class{Y}\cup\class{Y}'$. \\
\item $(\leftperp{(\rightperp{(\class{X}\cup\class{X}')})} , \class{Y}\cap \class{Y}')$ is the cotorsion pair cogenerated by $\class{X}\cup\class{X}'$.
\end{enumerate}  
\end{lemma}

\begin{proof}
To prove that the pair in (1) is indeed a cotorsion pair, it is enough to show $\class{X}\cap \class{X}' = \leftperp{(\class{Y}\cup\class{Y}')}$.  Clearly, $\class{X}\cap \class{X}' \subseteq \leftperp{(\class{Y}\cup\class{Y}')}$.  To show the reverse containment, note that $\class{Y} \subseteq \class{Y}\cup\class{Y}' \implies  \leftperp{\class{Y}} \supseteq \leftperp{(\class{Y}\cup\class{Y}')} \implies \class{X} \supseteq \leftperp{(\class{Y}\cup\class{Y}')}$. In the same way, $\class{X}' \supseteq \leftperp{(\class{Y}\cup\class{Y}')}$, and so 
$\class{X}\cap \class{X}' \supseteq \leftperp{(\class{Y}\cup\class{Y}')}$.
Statement (2) is similar. 
\end{proof}

\begin{theorem}[New models from injective ones]\label{theorem-cots-models}
Let $(\class{W},\class{R})$ and  $(\class{C},\class{F})$  both be complete cotorsion pairs in an exact category $(\class{A},\class{E})$ having enough injectives.
Then $(\class{W},\class{R})$ is an injective cotorsion pair with $\class{F}\subseteq \class{W}$  if and only if $\class{F}\cup\class{R}$ generates a Hovey triple,
$$\mathfrak{M} = (\class{C}, \class{W}, \rightperp{(\leftperp{(\class{F}\cup\class{R})})}).$$
Moreover, the following statements are equivalent and characterize the class   of fibrant objects: 
\begin{enumerate}
\item $M \in \rightperp{(\leftperp{(\class{F}\cup\class{R})})}$.
\item Any special $\class{W}$-precover of $M$ is in $\class{F}$.
\item There is a short exact sequence $R \rightarrowtail F \twoheadrightarrow M$ with $F \in \class{F}$ and $R\in\class{R}$.
\item $\Ext^1_{\class{E}}(N,M)=0$ for all modules $N\in \class{C}\cap \class{W}$.
\end{enumerate}
\end{theorem}

\begin{remark}
Note that the class $\rightperp{(\leftperp{(\class{F}\cup\class{R})})}$ of fibrant objects in the above Hovey triple is the smallest one possible containing all of the objects of both $\class{F}$ and $\class{R}$.  In the case that $\mathfrak{M}$ is hereditary we will see a nicer characterization of the fibrant objects   in Theorem~\ref{theorem-fibrant-objects}.
\end{remark}

\begin{proof}
($\impliedby$) Assume that  $\mathfrak{M}$ is a  Hovey triple. Then, by definition, $\class{W}$ is   thick and we must have $\class{F} = \rightperp{\class{C}} = \class{W}\cap  \rightperp{(\leftperp{(\class{F}\cup\class{R})})}$. In particular, $\class{F} \subseteq \class{W}$. Moreover, since $\class{F}$ must contain all injective objects, so does $\class{W}$. It follows easily, see~\cite[Proposition~2.21]{gillespie-book}(4), that  $(\class{W},\class{R})$ is an injective cotorsion pair. 

\noindent  ($\implies$)
Assume that $(\class{W},\class{R})$ is an injective cotorsion pair with $\class{F} \subseteq \class{W}$. 
By Lemma~\ref{lemma-cots}(1),  $\class{F}\cup\class{R}$ generates the  cotorsion pair, 
$(\class{C}\cap \class{W}, \rightperp{(\leftperp{(\class{F}\cup\class{R})})})$. 
Next we will show $\class{W}\cap  \rightperp{(\leftperp{(\class{F}\cup\class{R})})} = \class{F}$.
 First, the containment $\class{F} \subseteq \class{W}\cap \rightperp{(\leftperp{(\class{F}\cup\class{R})})}$ follows from the fact that both $\class{F} \subseteq \class{W}$ and $\class{F} \subseteq \class{F}\cup\class{R} \subseteq \rightperp{(\leftperp{(\class{F}\cup\class{R})})}$. To prove the reverse,  let $Z \in \class{W}\cap\rightperp{(\leftperp{(\class{F}\cup\class{R})})}$.  Since $(\class{C},\class{F})$ is a complete cotorsion pair, we may write a short exact sequence $Z \rightarrowtail F \twoheadrightarrow C$ where $F\in\class{F}$ and $C\in\class{C}$. Since $F, Z \in \class{W}$ and $\class{W}$ is thick,  we get  $C\in\class{W}$. Thus $C \in \class{C}\cap \class{W}  =  \leftperp{(\class{F}\cup\class{R})}$. As  $Z \in\rightperp{(\leftperp{(\class{F}\cup\class{R})})}$,  it follows that the previous short exact sequence splits. This means that  $Z$ is a direct summand of $F$ and hence $Z\in\class{F}$ too, completing the proof that $\class{W}\cap\rightperp{(\leftperp{(\class{F}\cup\class{R})})} = \class{F}$. 
 
 We can show directly that $(\class{C}\cap \class{W}, \rightperp{(\leftperp{(\class{F}\cup\class{R})})})$ is complete. To show it has enough injectives, let $M$ be given. We construct a pullback diagram
$$\begin{tikzcd}
 & F \arrow[equal]{r} \arrow[d, tail] & F \arrow[tail]{d} \\
M \arrow[tail]{r} \arrow[equal]{d} & P \arrow[two heads]{r} \arrow[two heads]{d} & C \arrow[two heads]{d} \\
M \arrow[tail]{r}  & R \arrow[two heads]{r} & W
\end{tikzcd}$$
as follows: The bottom row is obtained by using that $(\class{W},\class{R})$ has enough injectives, while the right column is obtained by using that $(\class{C},\class{F})$ has enough projectives. Then since $\class{F}\subseteq \class{W}$, we have $F,W\in\class{W}$, and hence $C \in \class{C}\cap\class{W}$.  We also have $R \in \class{R} \subseteq \rightperp{(\leftperp{(\class{F}\cup\class{R})})}$ and also $F \in \class{F} \subseteq \rightperp{(\leftperp{(\class{F}\cup\class{R})})}$, hence $P \in  \rightperp{(\leftperp{(\class{F}\cup\class{R})})}$.
 So the middle row of the diagram shows that our cotorsion pair has enough injectives. 
 
 A slightly different argument shows $(\class{C}\cap \class{W}, \rightperp{(\leftperp{(\class{F}\cup\class{R})})})$ has enough projectives. For this, we construct a pullback diagram
$$\begin{tikzcd}
F \arrow[equal]{r} \arrow[tail]{d} & F \arrow[tail]{d} &  \\
P \arrow[tail]{r} \arrow[two heads]{d} & C \arrow[two heads]{r} \arrow[two heads]{d} & M \arrow[equal]{d} \\
R \arrow[tail]{r} & W \arrow[two heads]{r} & M
\end{tikzcd}$$
as follows: The bottom row is obtained by using that $(\class{W},\class{R})$ has enough projectives, while the middle column is obtained by using that $(\class{C},\class{F})$ has enough projectives. Then again we can argue that $C \in \class{C}\cap\class{W}$,  and $P \in  \rightperp{(\leftperp{(\class{F}\cup\class{R})})}$, so that the middle row proves that our cotorsion pair has enough projectives. 

Since $\class{W}$ was already handed to us as a thick class, this completes the proof that $\mathfrak{M}$ is a Hovey triple.  
Lastly, we prove the stated characterizations of the fibrant objects:

\noindent (1)$\implies$(2). Let $M \in \rightperp{(\leftperp{(\class{F}\cup\class{R})})}$. Because $(\class{W},\class{R})$ is a complete cotorsion pair, we may  write a special $\class{W}$-precover sequence, $R \rightarrowtail W \twoheadrightarrow M$, with $W \in \class{W}$ and $R\in\class{R}\subseteq \rightperp{(\leftperp{(\class{F}\cup\class{R})})}$. Then $W\in \class{W}\cap \rightperp{(\leftperp{(\class{F}\cup\class{R})})} = \class{F}$.

\noindent (2)$\implies$(3) is trivial. 

\noindent (3)$\implies$(4).  Let $N\in \class{C}\cap \class{W}$. Then applying $\Hom_{\class{A}}(N,-)$ to such a given short exact sequence yields an exact sequence of abelian groups
$$\Ext^1_{\class{E}}(N,F) \xrightarrow{} \Ext^1_{\class{E}}(N,M) \xrightarrow{} \Ext^2_{\class{E}}(N,R).$$
We have $\Ext^1_{\class{E}}(N,F)=0$, since $N\in \class{C}$, and  $\Ext^2_{\class{E}}(N,R)=0$, since $N\in \class{W}$ and since  $(\class{W},\class{R})$ is hereditary, by~\cite[Corollary 2.17]{gillespie-book}. Therefore, $\Ext^1_{\class{E}}(N,M) =0$ for all $N\in \class{C}\cap \class{W}$.

\noindent (4)$\implies$(1). We already know that $\rightperp{(\class{C}\cap \class{W})}  =  \rightperp{(\leftperp{(\class{F}\cup\class{R})})}$.
\end{proof}

We now show that, in the  hereditary case, an object $M$ is fibrant in $\mathfrak{M}$ if and only if  $M=Z_{0}X$ for some exact $\Hom_{\class{A}}(\class{C}\cap\class{W}, -)$-acyclic  complex $X$ with each  $X_n \in \class{F}$.

\begin{theorem}[Hereditary condition and fibrant objects]\label{theorem-fibrant-objects}
The Hovey triple $\mathfrak{M} = (\class{C}, \class{W}, \rightperp{(\leftperp{(\class{F}\cup\class{R})})})$   of Theorem~\ref{theorem-cots-models}   is hereditary if and only if  $(\class{C},\class{F})$ is hereditary. In this case, the equivalent statements (1)--(4) characterizing  the fibrant objects  are also equivalent to these statements:
\begin{enumerate}\setcounter{enumi}{4}
\item There is a $\Hom_{\class{A}}(\class{C}\cap \class{W},-)$-exact short exact sequence  $M \rightarrowtail R \twoheadrightarrow F$ with $F \in \class{F}$ and $R\in\class{R}$.
\item  $M=Z_{0}X$ for some exact  complex $X$ with each  $X_n \in \class{F}$ and such that $\Hom_{\class{A}}(N,X)$ remains exact for any $N\in\class{C}\cap\class{W}$.
\end{enumerate}
\end{theorem}

\begin{proof}
Since $\class{W}$ is thick,  $\class{C}\cap\class{W}$ will be projectively resolving  whenever $\class{C}$ is projectively resolving. So the Hovey triple $\mathfrak{M} = (\class{C}, \class{W}, \rightperp{(\leftperp{(\class{F}\cup\class{R})})})$ is hereditary if and only if $(\class{C},\class{F})$ is hereditary; see~\cite[Corollary 2.17(2)]{gillespie-book}. 

For the remainder of the proof we  assume that $\mathfrak{M}$ is hereditary. We need to  prove statements (5) and (6) are also equivalent to the previous statements (1)--(4). 

\noindent (5)$\implies$(4). 
Assume there is a $\Hom_{\class{A}}(\class{C}\cap \class{W},-)$-exact short exact sequence $M \rightarrowtail R \twoheadrightarrow F$ with $F \in \class{F}$ and $R\in\class{R}$. Then given any $N\in \class{C}\cap \class{W}$, we apply $\Hom_{\class{A}}(N,-)$ and use that $\Ext^1_{\class{E}}(N,R)=0$ to conclude that $\Ext^1_{\class{E}}(N,M)=0$ too.

\noindent (1)$\implies$(5).  Assume $M \in \rightperp{(\leftperp{(\class{F}\cup\class{R})})}$, and use 
that $(\class{W}, \class{R})$ has enough injectives to write  a short exact sequence $M \rightarrowtail R \twoheadrightarrow W$   with $R\in  \class{R}$ and $W\in\class{W}$.  This short exact sequence is indeed $\Hom_{\class{A}}(\class{C}\cap \class{W},-)$-exact   because  $\class{C}\cap \class{W} = \leftperp{(\class{F}\cup\class{R})}$. Moreover, since $\mathfrak{M}$ is hereditary, we get $W \in \rightperp{(\leftperp{(\class{F}\cup\class{R})})}$. Therefore we have $W\in  \class{W}\cap\rightperp{(\leftperp{(\class{F}\cup\class{R})})} = \class{F}$. 

Now we show that statements (1)--(5) imply (6).
First, using (3), let $R \rightarrowtail X_1 \twoheadrightarrow M$ be a short exact sequence having $X_1 \in \class{F}$ and $R\in\class{R}$. Then $\Ext^1_{\class{E}}(N,R) = 0$ for any $N\in\class{W}$, so the short exact sequence remains exact after applying the functor $\Hom_{\class{A}}(N,-)$ for any $N\in\class{W}$. Since $R\in\class{R} \subseteq \rightperp{(\leftperp{(\class{F}\cup\class{R})})}$, the equivalent statements (1)--(5) apply to $R$ too. So  we may iterate the construction to obtain  an exact complex
  $$\cdots \xrightarrow{} X_3\xrightarrow{} X_2\xrightarrow{} X_1  \twoheadrightarrow M$$ with each $X_n \in \class{F}$  which remains exact upon applying the functor $\Hom_{\class{A}}(N,-)$ for any $N\in\class{W}$, so in particular  for any $N\in\class{C}\cap\class{W}$. On the other hand, by  repeatedly using that $(\class{C}, \class{F})$ has enough injectives,  we may construct an exact complex 
  $$M \rightarrowtail X_{0}\xrightarrow{} X_{-1}\xrightarrow{} X_{-2}  \xrightarrow{} \cdots$$ having  each $X_n \in \class{F}$. 
 By the hereditary condition, we get that each cokernel of the complex is again in $\rightperp{(\leftperp{(\class{F}\cup\class{R})})}$. In particular, 
the complex will remain exact upon applying   $\Hom_{\class{A}}(N,-)$ for any $N\in\class{C}\cap\class{W}$.  This competes the proof that statements (1)--(5) imply (6).

To finish,  we show that (6) implies (4). So assume that $M=Z_{0}X$ for some exact  complex $X$ with each  $X_n \in \class{F}$ and such that $\Hom_{\class{A}}(N,X)$ remains exact for all $N\in\class{C}\cap\class{W}$.  Letting $N\in \class{C}\cap \class{W}$, we need to show $\Ext^1_{\class{E}}(N,M)=0$. Applying $\Hom_{\class{A}}(N,-)$ to the short exact sequence $M = Z_0X\rightarrowtail X_0 \twoheadrightarrow Z_{-1}X$ yields another short exact sequence. Moreover, since $X_0\in \class{F} \subseteq  \rightperp{(\leftperp{(\class{F}\cup\class{R})})} = \rightperp{(\class{C}\cap\class{W})}$, we have
$\Ext^1_{\class{E}}(N, X_0) =0$, and so we get an exact sequence of abelian groups
$$0 \xrightarrow{} \Hom_{\class{A}}(N, M) \xrightarrow{} \Hom_{\class{A}}(N, X_0) \xrightarrow{} \Hom_{\class{A}}(N, Z_{-1}X) \xrightarrow{} \Ext^1_{\class{E}}(N, M) \xrightarrow{} 0$$ 
with the homomorphism $\Hom_{\class{A}}(N, Z_{-1}X) \xrightarrow{} \Ext^1_{\class{E}}(N, M)$ being the zero map. 
It follows that  $\Ext^1_{\class{E}}(N, M)=0$, proving statement (4). 
\end{proof}

\subsection{Projective duals}\label{subsec-projective-duals} We have constructed statements and proofs above so that they  allow for the following duals.

\begin{theorem}[New models from projective ones]\label{theorem-cots-models-proj}
Let $(\class{Q},\class{W})$ and  $(\class{C},\class{F})$  both be complete cotorsion pairs in an exact category $(\class{A},\class{E})$ having enough projectives.
Then $(\class{Q},\class{W})$ is an projective cotorsion pair with $\class{C}\subseteq \class{W}$  if and only if $\class{C}\cup\class{Q}$ cogenerates a Hovey triple,
$$\mathfrak{M} = (\leftperp{(\rightperp{(\class{C}\cup\class{Q})})}, \class{W}, \class{F}).$$
Moreover, the following statements are equivalent and characterize the class   of cofibrant objects: 
\begin{enumerate}
\item $M \in \leftperp{(\rightperp{(\class{C}\cup\class{Q})})}$.
\item Any special $\class{W}$-preenvelope of $M$ is in $\class{C}$.
\item There is a short exact sequence $M \rightarrowtail C \twoheadrightarrow Q$ with $C \in \class{C}$ and $Q\in\class{Q}$.
\item $\Ext^1_{\class{E}}(M,N)=0$ for all modules $N\in \class{W}\cap \class{F}$.
\end{enumerate}
\end{theorem}

\begin{theorem}[Hereditary condition and cofibrant objects]\label{theorem-cofibrant-objects}
The Hovey triple $\mathfrak{M} = (\leftperp{(\rightperp{(\class{C}\cup\class{Q})})}, \class{W}, \class{F})$    of Theorem~\ref{theorem-cots-models-proj}   is hereditary if and only if  $(\class{C},\class{F})$ is hereditary. In this case, the equivalent statements (1)--(4) characterizing  the cofibrant objects  are also equivalent to these statements:
\begin{enumerate}\setcounter{enumi}{4}
\item There is a $\Hom_{\class{A}}(-, \class{W}\cap \class{F})$-exact short exact sequence  $C \rightarrowtail Q \twoheadrightarrow M$ with $C \in \class{C}$ and $Q\in\class{Q}$.
\item  $M=Z_{0}X$ for some exact  complex $X$ with each  $X_n \in \class{C}$ and such that $\Hom_{\class{A}}(X, N)$ remains exact for any $N\in\class{W}\cap\class{F}$.
\end{enumerate}
\end{theorem}

\begin{example}\label{example-flat-cot}
Let $R$ be a ring. 
Take $(\class{Q},\class{W})$ in Theorem~\ref{theorem-cots-models-proj} to be the projective cotorsion pair  where $\class{Q}$ is the class of all \emph{Gorenstein AC-projective} $R$-modules, from \cite[Section~8]{bravo-gillespie-hovey}.  Take $(\class{C},\class{F})$ to be the complete hereditary cotorsion pair where $\class{C}$ is the class of all \emph{level} $R$-modules,  from \cite[Section~2]{bravo-gillespie-hovey}.  The hypotheses of Theorems~\ref{theorem-cots-models-proj}/\ref{theorem-cofibrant-objects} are met and we get a (cofibrantly generated) abelian model structure, 
$$\mathfrak{M} = (\leftperp{(\rightperp{(\class{C}\cup\class{Q})})}, \class{W}, \class{F}).$$
A module $M$ is cofibrant in $\mathfrak{M}$ if and only if 
 $M=Z_{0}L$ for some exact  complex $L$ of level modules such that $\Hom_{\class{A}}(L, N)$ remains exact for any $N\in\class{W}\cap\class{F}$.

 In the case that $R$ is (right) coherent, the level (left) modules coincide with the flat modules. Moreover,  the Gorenstein AC-projectives in this case are precisely the \emph{Ding projective} modules, or equivalently, the \emph{PGF-modules}. It follows from  \cite[Theorem~4.11(3)]{saroch-stovicek-G-flat} that $\leftperp{(\rightperp{(\class{C}\cup\class{Q})})}$ equals the class of all Gorenstein flat modules.  
 In particular, an exact complex $X$ of flat modules is \emph{F-totally acyclic} (i.e. $I\tensor_R X$ is exact for all injective right $R$-modules) if and only if $\Hom_R(X,N)$ remains exact for all cotorsion modules in $\class{W}$. 
 \end{example}


\section{Existence of non-hereditary models for derived categories}
\label{section-non-hereditary}

This brief stand alone section is an application of Theorem~\ref{theorem-cots-models}, showing the existence of non-hereditary model structures for the derived category of  a ring $R$.  Throughout this section, we let $\rmod$ denote the category of (left) $R$-modules over a ring $R$, and $\ch$ denote the corresponding category of chain complexes. 

In~\cite{gillespie} we described a recipe for lifting cotorsion pairs on $\rmod$ to cotorsion pairs yielding potential model structures on $\ch$. Using the notation there, any cotorsion pair  $(\class{A},\class{B})$ in $\rmod$ lifts to two cotorsion pairs, $(\dgclass{A},\tilclass{B})$ and $(\tilclass{A},\dgclass{B})$, in $\ch$. Here, $\tilclass{A}$  (resp. $\tilclass{B}$) is the class of all exact complexes with cycles in $\class{A}$ (resp. $\class{B}$). The other two classes are, essentially by definition, their respective Ext-orthogonal classes. 
Letting $\class{E}$ denote the class of all exact complexes, it is shown in~\cite[Corollary~3.13]{gillespie} that the  cotorsion pair  $(\class{A},\class{B})$ is hereditary if and only if $\dgclass{A}\cap\class{E} = \tilclass{A}$ and/or $\class{E}\cap\dgclass{B} = \tilclass{B}$. 
In this case, if the cotorsion pairs  $(\dgclass{A},\tilclass{B})$ and $(\tilclass{A},\dgclass{B})$ are complete, then they are the two cotorsion pairs corresponding to  a single abelian model structure, namely $(\dgclass{A}, \class{E}, \dgclass{B})$, on $\ch$. The following result shows that, in the non-hereditary case,  each of these cotorsion pairs is actually part of a \emph{different} abelian model structure. 

In what follows,  $\dgclass{P}$ denotes the well-known class of all \emph{DG-projective} chain complexes, while $\dgclass{I}$ denotes the dual class of all  \emph{DG-injective} complexes. 

\begin{theorem}[Non-hereditary models for derived categories]\label{theorem-non-hereditary-derived-models}
Let $R$ be a ring and $(\class{A},\class{B})$ be a  (not necessarily hereditary)  cotorsion pair of $R$-modules, cogenerated by a set. 
Then $(\dgclass{A},\tilclass{B})$ and $(\tilclass{A},\dgclass{B})$ are both complete cotorsion pairs in $\ch$,  and each is part of a cofibrantly generated abelian model structure for $\class{D}(R)$, the derived category of $R$:
$$\textnormal{(i)} \ \ (\dgclass{A}, \class{E}, \rightperp{(\leftperp{(\tilclass{B}\cup\dgclass{I})})})   \ \ \ \ \ \ \ \ \textnormal{(ii)} \ \  (\leftperp{(\rightperp{(\tilclass{A}\cup\dgclass{P})})}, \class{E},\dgclass{B}).$$
The model structures in $\textnormal{(i)}$ and $\textnormal{(ii)}$ are equal  if and only if $(\class{A},\class{B})$ is hereditary, in which case they coincide with the usual Hovey triple,  $(\dgclass{A}, \class{E},\dgclass{B})$. 
\end{theorem}

\begin{proof}
As already noted, the  two cotorsion pairs $(\dgclass{A},\tilclass{B})$ and $(\tilclass{A},\dgclass{B})$ exist. Since $(\class{A},\class{B})$ is cogenerated by a set, it follows that $(\dgclass{A},\tilclass{B})$  is also cogenerated by a set; see \cite[Prop. 3.8]{gillespie-quasi-coherent} or \cite[Prop. 10.37]{gillespie-book}. On the other hand, it follows from~\cite[Prop.~1.7/Prop.~2.9(1)/Them.~4.2(1)]{stovicek-deconstructible} that $(\tilclass{A},\dgclass{B})$ is also cogenerated by a set. So the two cotorsion pairs are complete. 
It is also well known that $(\dgclass{P},\class{E})$ is a projective cotorsion pair, while $(\class{E},\dgclass{I})$ is an injective  cotorsion pair (and again each is cogenerated by a set). So by 
Theorem~\ref{theorem-cots-models} we get an abelian model structure, 
$ (\dgclass{A}, \class{E}, \rightperp{(\leftperp{(\tilclass{B}\cup\dgclass{I})})})$, and by Theorem~\ref{theorem-cots-models-proj}  we get another abelian model structure,  $(\leftperp{(\rightperp{(\tilclass{A}\cup\dgclass{P})})}, \class{E},\dgclass{B})$.
Since we have equality of the cotorsion pairs,
 $$(\dgclass{A},\tilclass{B}) = (\dgclass{A}, \class{E}\cap\rightperp{(\leftperp{(\tilclass{B}\cup\dgclass{I})})}),$$ and 
 $$(\tilclass{A}, \dgclass{B}) =  (\leftperp{(\rightperp{(\tilclass{A}\cup\dgclass{P})})}\cap\class{E},\dgclass{B}),$$
 we deduce that the model structures in (i) and (ii) are equal if and only if 
 $$\tilclass{B} = \class{E}\cap\dgclass{B}   \text{\ \ \ \ and\ \ \ \ } \tilclass{A} = \dgclass{A}\cap\class{E}.$$
But by  \cite[Corollary~3.13]{gillespie}, this is the case if and only if  the  cotorsion pair  $(\class{A},\class{B})$ is hereditary.
\end{proof}
 
One can check that $\rightperp{(\leftperp{(\tilclass{B}\cup\dgclass{I})})}$ is precisely the class of all chain complexes $X$  having each $X_n\in\class{B}$ and with the property that any chain map $A \xrightarrow{} X$ with $A \in \dgclass{A}\cap\class{E}$   is null homotopic. 
The analogous characterization holds for  $\leftperp{(\rightperp{(\tilclass{A}\cup\dgclass{P})})}$.
 
\begin{example}
Probably the most well-known non-hereditary cotorsion pair involves the class of FP-injective modules over a non-coherent ring. In more detail, let $\class{S}$ be a set of isomorphism representatives for the class of all finitely presented  $R$-modules. $\class{S}$ cogenerates a complete cotorsion pair which we denote by $(\class{FP},\class{FI})$. The modules in $\class{FI}$ are precisely the \emph{FP-injective} modules, while the modules in $\class{FP}$ are  called \emph{FP-projective}. It is known that  $R$ is  (left) coherent if and only if $(\class{FP},\class{FI})$ is hereditary; see~\cite[Appendix~B]{stovicek-purity}. So the classes $\dgclass{PF}$ and $\dgclass{FI}$ are always the (co)fibrant objects of abelian models for the derived category, but they are two distinct model structures whenever $R$ is not (left) coherent. 
\end{example}


\section{Bifibrant objects and  totally acyclic complexes}\label{subsec-tech-acyclic}

To prepare for our second application of Theorems~\ref{theorem-cots-models} and~\ref{theorem-fibrant-objects}, we  return to the general setting of Section~\ref{section-merging-models}. So throughout this section, we assume $(\cat{A},\class{E})$ is an exact category with enough injectives,  $(\class{C},\class{F})$ is a complete hereditary  cotorsion pair, and  $(\class{W},\class{R})$ is an injective cotorsion pair with $\class{F}\subseteq \class{W}$. Then $\mathfrak{M} = (\class{C}, \class{W}, \rightperp{(\leftperp{(\class{F}\cup\class{R})})})$ denotes the corresponding Hovey triple of Theorems~\ref{theorem-cots-models} and~\ref{theorem-fibrant-objects}. Our goal here is to record some general conditions, to be applied later, guaranteeing a particularly nice description of the homotopy category, $\textnormal{Ho}(\mathfrak{M})$. This description  appears in Theorem~\ref{theorem-stable-CF} and  uses the following terminology. 


\begin{definition}\label{def-Goren-CF}
We set the following terminology for exact complexes  having components in $\class{C}\cap\class{F}$, the core of the cotorsion pair $(\class{C}, \class{F})$. 
\begin{enumerate} 
\item By a \textbf{\emph{totally acyclic complex of  $(\boldsymbol{\class{C}\cap\class{F}})$-objects}} we mean an $\class{E}$-exact complex $X$ having each $X_n \in \class{C}\cap\class{F}$ and satisfying that both $\Hom_{\cat{A}}(N,X)$ and $\Hom_{\cat{A}}(X,N)$ are exact complexes of abelian groups whenever $N \in \class{C}\cap\class{F}$. 
\item By a \textbf{\emph{Gorenstein $(\boldsymbol{\class{C}\cap\class{F}})$-object}}
 we mean an object $M = Z_0X$ equaling  the 0-cycle  of some 
totally acyclic complex of  $(\class{C}\cap\class{F})$-objects.
\end{enumerate}
We let $\class{GCF}$ denote the class of all Gorenstein $(\class{C}\cap\class{F})$-objects. 
\end{definition}

Note that the following proposition immediately implies that all bifibrant objects of $\mathfrak{M}$ are in the class $\class{GCF}$  of  all Gorenstein $(\class{C}\cap\class{F})$-objects. 

\begin{proposition}\label{prop-bifibrant-acyclicity}
An object $M$ is bifibrant, meaning $M \in \class{C} \cap \rightperp{(\leftperp{(\class{F}\cup\class{R})})}$, if and only if, 
$M=Z_{0}X$ for some $\class{E}$-exact  complex $X$ satisfying that each  $X_n \in \class{C}\cap\class{F}$, and  $\Hom_{\class{A}}(N,X)$ is an exact complex of abelian groups for any $N\in\class{C}\cap\class{W}$, and also  $\Hom_{\class{A}}(X,N)$ is an exact complex of abelian groups for any $N\in\class{F}$.
\end{proposition}

\begin{proof}
\noindent ($\impliedby$)  Suppose $M=Z_{0}X$ for such an exact complex $X$. Then we certainly have  $M \in \rightperp{(\leftperp{(\class{F}\cup\class{R})})}$, by Theorem~\ref{theorem-fibrant-objects}(6).
To show $M \in \class{C}$, we let $N\in \class{F}
$, and we  show $\Ext^1_{\class{E}}(M,N)=0$. By our assumption, applying $\Hom_{\class{A}}(-, N)$ to the short exact sequence $Z_1X\rightarrowtail X_1 \twoheadrightarrow M$ yields another short exact sequence. Moreover, since $X_1\in \class{C} =  \leftperp{\class{F}}$, we have
$\Ext^1_{\class{E}}(X_1, N) =0$, and so we get an exact sequence of abelian groups
$$0 \xrightarrow{} \Hom_{\class{A}}(M, N) \xrightarrow{} \Hom_{\class{A}}(X_1, N) \xrightarrow{} \Hom_{\class{A}}(Z_{1}X, N) \xrightarrow{} \Ext^1_{\class{E}}(M, N) \xrightarrow{} 0$$ 
with the homomorphism $\Hom_{\class{A}}(Z_{1}X, N) \xrightarrow{} \Ext^1_{\class{E}}(M, N)$ being the zero map. 
It follows that  $\Ext^1_{\class{E}}(M, N)=0$, proving $M\in\class{C}$. Therefore,  $M \in \class{C} \cap \rightperp{(\leftperp{(\class{F}\cup\class{R})})}$.

\noindent ($\implies$) We return  again to  the proof of Theorem~\ref{theorem-fibrant-objects} and follow the argument given  there showing part (6).
It starts by taking a short exact sequence  $R \rightarrowtail X_1 \twoheadrightarrow M$  having $X_1 \in \class{F}$ and $R\in\class{R}$.   We now also use that $(\class{C}, \class{F})$ has enough projectives to write another short exact sequence 
 $F \rightarrowtail X'_1 \twoheadrightarrow X_1$  having $X'_1 \in \class{C}$ and $F\in\class{F}$.  
 Note that $X'_1$ is necessarily in $\class{C}\cap\class{F}$. 
Form the pullback diagram:
$$\begin{tikzcd}
F \arrow[equal]{r} \arrow[tail]{d} & F \arrow[tail]{d} &  \\
R' \arrow[tail]{r} \arrow[two heads]{d} & X'_1 \arrow[two heads]{r} \arrow[two heads]{d} & M \arrow[equal]{d} \\
R \arrow[tail]{r} & X_1 \arrow[two heads]{r} & M.
\end{tikzcd}$$
Since $F\in\class{F}$ and $R\in \class{R}$, we get that the extension, $R'$, is in the class
$\rightperp{(\leftperp{(\class{F}\cup\class{R})})}$. Moreover, since $M, X'_1 \in\class{C}$ and $(\class{C}, \class{F})$ is hereditary, we get  $R' \in\class{C} \cap\rightperp{(\leftperp{(\class{F}\cup\class{R})})}$. 
Noting that the middle row of the diagram  has  $R' \in\class{C} \cap\rightperp{(\leftperp{(\class{F}\cup\class{R})})}$ and $X'_1\in\class{C}\cap\class{F}$, we can iterate this construction to obtain  an exact complex
  $$\cdots \xrightarrow{} X'_3\xrightarrow{} X'_2\xrightarrow{} X'_1  \twoheadrightarrow M$$ with each $X'_n \in \class{C}\cap\class{F}$, and with each cokernel in $\class{C}\cap\rightperp{(\leftperp{(\class{F}\cup\class{R})})}$.  In particular, the complex remains exact upon  applying  $\Hom_{\cat{A}}(N, -)$ for any $N\in\class{C}\cap\class{W}$, and upon applying  $\Hom_{\cat{A}}(-, N)$ for any $N\in\class{F}$.
      On the other hand, since $M\in\class{C}\cap\rightperp{(\leftperp{(\class{F}\cup\class{R})})}$, by using that $(\class{C}, \class{F})$ has enough injectives, the construction of the exact complex
  $$M \rightarrowtail X_{0}\xrightarrow{} X_{-1}\xrightarrow{} X_{-2}  \xrightarrow{} \cdots,$$ will again have each $X_n \in \class{C}\cap\class{F}$, and, by the hereditary condition, each cokernel will again be in $\class{C}\cap\rightperp{(\leftperp{(\class{F}\cup\class{R})})}$.
 So  again, the complex remains exact upon  applying  $\Hom_{\cat{A}}(N, -)$ for any $N\in\class{C}\cap\class{W}$, and upon  applying  $\Hom_{\cat{A}}(-, N)$ for any $N\in\class{F}$.
Pasting the two complexes together at $M$  we obtain the desired exact complex $X$.
\end{proof} 

In light of the fact that $(\class{C}, \class{F})$ is a complete hereditary cotorsion pair, it makes good sense to define  the  \emph{$\class{F}$-dimension} of an object $N$ as follows: If there is a non-negative integer $n$ and an exact sequence 
$$N  \rightarrowtail F^0 \xrightarrow{} F^1 \xrightarrow{} \cdots  \xrightarrow{} F^{n-1} \twoheadrightarrow F^n$$
with each $F^i \in \class{F}$, then we say that $N$ has  \textbf{finite $\class{F}$-dimension}, and write $\class{F}\textnormal{-dim}(N) \leq n$. If $d$ is the least integer for which such an exact sequence exists,  then this is the  \textbf{$\class{F}$-dimension} of $N$, written $\class{F}\textnormal{-dim}(N) = d$.   If no such integer $n$ exists then we write $\class{F}\textnormal{-dim}(N) = \infty$. 
Using standard dimension shifting arguments  one can use the hereditary condition to check that $\class{F}\textnormal{-dim}(N) \leq n<\infty$ implies that $\Ext^{n+i}_{\class{E}}(C,N) = 0$ for all $i\geq 1$ and all $C\in\class{C}$. We use this in the proof of the next lemma.

\begin{lemma}\label{lemma-CF-dimensions}
An object $N\in\class{C}\cap\class{W}$  has finite $\class{F}$-dimension,  $\class{F}\textnormal{-dim}(N) \leq n$,  if and only if, there exists an exact sequence 
$$N  \rightarrowtail F^0 \xrightarrow{} F^1 \xrightarrow{} \cdots  \xrightarrow{} F^{n-1} \twoheadrightarrow F^n$$
with each $F^i \in \class{C}\cap\class{F}$. 
\end{lemma}

\begin{proof}
($\implies$) We start by constructing a coresolution
$$N  \rightarrowtail F^0 \xrightarrow{} F^1 \xrightarrow{} F^2 \xrightarrow{} \cdots $$
of $N$ by successively using that the cotorsion pair $(\class{C},\class{F})$ has enough injectives. 
Since $N\in\class{C}$, each cycle (kernel) is in $\class{C}$ and hence each $F^i \in \class{C}\cap\class{F}$. 
If $\class{F}\textnormal{-dim}(N) \leq n$, then by the standard dimension shifting argument one can see that the $n$th cosyzygy  (cokernel) of the coresolution must be in $\class{F}$.  So it in $\class{C}\cap\class{F}$ and we are done.

\noindent ($\impliedby$) Suppose $N$ has such a coresolution. Then because $\class{F}\subseteq \class{W}$, and because $\class{W}$ satisfies the 2 out of 3 property, we get $N\in\class{W}$.  It is also  clear that  $N \in \class{C}$ since  the cotorsion pair $(\class{C}, \class{F})$ is hereditary.  
\end{proof}

\begin{lemma}\label{lemma-exact-complexes-CW-totally}
Assume each object of  $N\in\class{C}\cap\class{W}$  has finite $\class{F}$-dimension. Then the following are equivalent for a chain  complex $X$ having each $X_n \in  \class{F}$. 
\begin{enumerate}
\item $\Hom_{\cat{A}}(N,X)$ is exact for all $N \in \class{C}\cap\class{W}$.  
\item $\Hom_{\cat{A}}(N,X)$ is exact for all $N \in \class{C}\cap\class{F}$. 
\end{enumerate}
\end{lemma}

\begin{proof}
Certainly (1) implies (2) since $\class{F}\subseteq \class{W}$. To prove (2) implies (1), we use Lemma~\ref{lemma-CF-dimensions} to write an exact sequence 
$$N  \rightarrowtail F^0 \xrightarrow{} F^1 \xrightarrow{} \cdots  \xrightarrow{} F^{d-1} \twoheadrightarrow F^d$$
having each $F^i \in \class{C}\cap\class{F}$.   For each $i = 0, 1, \cdots , d-1$, set $N^i := \Ker(F^i \xrightarrow{} F^{i+1})$. Then each $N^i$ must be in $\class{C}$, because of the hereditary condition, and we have $N^0 = N$. Since $\Ext^1_{\class{E}}(N^i, X_n) = 0$, applying $\Hom_{\cat{A}}(-, X)$ to each short exact sequence $N^i  \rightarrowtail F^i \twoheadrightarrow N^{i+1}$ results in a short exact sequence of chain complexes of abelian groups
$$0 \xrightarrow{} \Hom_{\cat{A}}(N^{i+1}, X) \xrightarrow{} \Hom_{\cat{A}}(F^{i}, X) \xrightarrow{} \Hom_{\cat{A}}(N^{i}, X) \xrightarrow{} 0.$$ 
So since $\Hom_{\cat{A}}(F^{i}, X)$ is exact for any $F^i$, it follows from the 2 out of 3 property that $\Hom_{\cat{A}}(N^i, X)$ is exact for all $i = 0, 1, \cdots , d-1$.
\end{proof}


Recall that $\class{GCF}$ denotes the class of all Gorenstein $(\class{C}\cap\class{F})$-objects as spelled out  in Definition~\ref{def-Goren-CF}. 

\begin{theorem}\label{theorem-stable-CF}
Suppose that the Hovey triple $\mathfrak{M} = (\class{C}, \class{W}, \rightperp{(\leftperp{(\class{F}\cup\class{R})})})$   of Theorem~\ref{theorem-cots-models} exists,  and is hereditary, and that  these  two conditions hold:
\begin{enumerate}
\item  Each object $N\in\class{C}\cap\class{W}$  has finite $\class{F}$-dimension. 
\item Any $\class{E}$-exact chain complex $X$, having all  of its   components   $X_n\in\class{C}\cap\class{F}$, necessarily has each cycle $Z_nX\in\class{C}$. 
\end{enumerate}
Then we have $\class{C}\cap\rightperp{(\leftperp{(\class{F}\cup\class{R})})} = \class{GCF}$, the class of all Gorenstein $(\class{C}\cap\class{F})$-objects. Therefore, we may identify the homotopy category as 
 $$\textnormal{Ho}(\mathfrak{M}) \simeq \textnormal{St}(\class{GCF}),$$ 
 where $\textnormal{St}(\class{GCF})$ denotes the stable category of  the Frobenius category of all Gorenstein $(\class{C}\cap\class{F})$-objects. 
\end{theorem}

\begin{remark}
Condition (2) can be  stated in the language of \emph{periodicity} as in~\cite{bazzoni-cortes-estrada} and~\cite{bazzoni-hrbek-positselski-fp-projective-periodicity}. 
\end{remark}

\begin{proof}
We just need to show  $\class{C}\cap\rightperp{(\leftperp{(\class{F}\cup\class{R})})} = \class{GCF}$.

\noindent ($\subseteq$) $M \in \class{C} \cap \rightperp{(\leftperp{(\class{F}\cup\class{R})})}$. Then  by Proposition~\ref{prop-bifibrant-acyclicity},
$M=Z_{0}X$ for some $\class{E}$-exact  complex $X$ satisfying that each  $X_n \in \class{C}\cap\class{F}$, and  $\Hom_{\class{A}}(N,X)$ is exact for any $N\in\class{C}\cap\class{W}$, and also  $\Hom_{\class{A}}(X,N)$ is exact for any $N\in\class{F}$. In particular, $X$ is a totally acyclic complex of $(\class{C}\cap\class{F})$-objects. So $M \in \class{GCF}$.

\noindent ($\supseteq$)  Suppose $M \in \class{GCF}$. Write $M=Z_0X$ for some totally acyclic complex $X$  of $(\class{C}\cap\class{F})$-objects.  By  the periodicity condition, i.e., assumption (2),  we get that $M \in \class{C}$. We also get by Lemma~\ref{lemma-exact-complexes-CW-totally}  that $\Hom_{\cat{A}}(N, X)$ is an exact complex for each $N\in\class{C}\cap\class{W}$.  Therefore, we also get $M\in \rightperp{(\leftperp{(\class{F}\cup\class{R})})}$, by Theorem~\ref{theorem-fibrant-objects}(6).
\end{proof}

One could again formulate duals of the results of this section.
 In particular, Theorem~\ref{theorem-stable-CF} has a dual.


\section{Weakly Gorenstein $\class{B}$-injective modules}\label{section-weakly-B}

Now we again let $R$ be a ring and let $\rmod$ denote the category of (left) $R$-modules. 
Throughout this section, \textbf{\emph{$\boldsymbol{\class{B}}$  denotes a class of $\boldsymbol{R}$-modules  that contains all of the injective modules}}. 
The following notion of a \emph{Gorenstein $\class{B}$-injective} module was studied in~\cite[Definition 16]{gillespie-iacob-duality-pairs}. The weaker variant below is based on more recent work of Iacob (\cite{iacob-weakly-Ding-1}, \cite{iacob-weakly-Ding-2}, and \cite{iacob-weakly-Ding-3}) in which she studies the \emph{weakly Ding injective} modules from~\cite{Goren-FP-inj}.

\begin{definition}\label{Defs-relative-G-inj}
As in Section~\ref{subsec-Hom-acyclicity}, we say that a chain complex $X$ of $R$-modules is \emph{$\Hom_R(\class{B},-)$-acyclic} if $\Hom_R(B,X)$ is an exact complex of abelian groups for all $B \in \class{B}$. If $X$ itself is also exact we say $X$ is an \emph{exact $\Hom_R(\class{B},-)$-acyclic} complex. 
\begin{enumerate}
\item An $R$-module $Z$ is \textbf{\emph{Gorenstein $\class{B}$-injective}} if  $Z=Z_{0}E$ for some exact $\Hom_R(\class{B},-)$-acyclic complex $E$ of $R$-modules with each $E_n$ injective.
\item An $R$-module $Z$ is \textbf{\emph{weakly Gorenstein $\class{B}$-injective}} if  $Z=Z_{0}E$ for some exact $\Hom_R(\class{B},-)$-acyclic complex  $E$ of $R$-modules with each  $E_n \in \class{B}$.
\end{enumerate}
\end{definition}

\begin{notation}\label{notation-relative-G-inj}
We let $\class{GI}_{\class{B}}$ denote the class of all Gorenstein $\class{B}$-injective $R$-modules, and we let $\textnormal{w}\class{GI}_{\class{B}}$  denote the class of all weakly Gorenstein $\class{B}$-injective $R$-modules. In general, we have $\class{GI}_{\class{B}} \subseteq \textnormal{w}\class{GI}_{\class{B}}$, and, $\class{B}\subseteq \textnormal{w}\class{GI}_{\class{B}}$. Note too that if $\class{GI}$ denotes the class of all the usual Gorenstein injectives, then we have $\class{GI}_{\class{B}} \subseteq \class{GI}$,   with equalities $\class{GI} = \class{GI}_{\class{B}} =  \textnormal{w}\class{GI}_{\class{B}}$ in the case that $\class{B}$ is nothing more than the class of all injective modules. 
\end{notation}

The next few lemmas are independent of one another, but all will be used to prove the main theorem of this section. All of the results and their  proofs below are based on similar results of Iacob that can be found througout~\cite{iacob-weakly-Ding-1, iacob-weakly-Ding-2, iacob-weakly-Ding-3}, but mainly \cite{iacob-weakly-Ding-3}.

\begin{lemma}\label{lemma-weakly-covers}
Assume $\class{B}$ is a covering class and closed under extensions. Let $L$ be a direct summand of a weakly Gorenstein $\class{B}$-injective module, $Z$. Then its $\class{B}$-cover  $\psi \colon C \xrightarrow{} L$ must be surjective with $\Ker{\psi} \in \rightperp{\class{B}}$. In particular, it's true for $L = Z$.
\end{lemma}

\begin{proof}
By the definition of $Z$, there is a surjection $e \colon E \xrightarrow{} Z$ with $E \in \class{B}$. Consider the composite of surjections, $E \xrightarrow{e} Z \xrightarrow{\pi} L$,   where $\pi$ is projection onto the summand $L$. It lifts over the cover, as $\pi e = \psi t$ for some $t$. Since $\pi e$ is surjective, so is $\psi$. Since $\class{B}$ is assumed to be closed under extensions, Wakamatsu's Lemma \cite[Corollary 7.2.3]{enochs-jenda-book} applies; it asserts that  $\Ker{\psi} \in \rightperp{\class{B}}$. 
(Our blanket assumption that $\class{B}$ contains all injectives is not needed.)
\end{proof}

\begin{lemma}\label{lemma-injectives}
Assume that $\class{B}$  is closed under cokernels of monomorphisms. Then the class 
$\class{B}\cap\rightperp{\class{B}}$ is equal to the class of all injective modules. 
\end{lemma}

\begin{proof}
Every injective module is in $\class{B}\cap\rightperp{\class{B}}$. Conversely, say $X\in\class{B}\cap\rightperp{\class{B}}$, and write a short exact sequence 
$$0\xrightarrow{} X  \xrightarrow{} I \xrightarrow{} I/X \xrightarrow{} 0$$ 
where $I$ is an injective module. The hypotheses imply that $I/X\in\class{B}$. So we get $\Ext^1_R(I/X,X)=0$. This means the short exact sequence must split and hence $X$, being a direct summand of $I$, is itself an injective module. 
\end{proof}

\begin{lemma}\label{lemma-cover-sums}
Assume we have a short exact sequence of $R$-modules 
$$0 \xrightarrow{} K \xrightarrow{}    C  \xrightarrow{\psi} L \xrightarrow{} 0
$$
where $\psi$ is a $\class{B}$-cover. Then  for any $B \in\class{B}$, we have a short exact sequence
$$0 \xrightarrow{} K \xrightarrow{}    C\Oplus B \xrightarrow{\psi \oplus1_B} L\Oplus B \xrightarrow{} 0
$$
and the map $\psi \oplus1_B \colon C\Oplus B \xrightarrow{} L\Oplus B$ is also a  $\class{B}$-cover with $\Ker{(\psi\oplus 1_B)} \cong K$.
\end{lemma}

\begin{proof}
The second sequence is just the direct sum of the first with the sequence $0 \xrightarrow{} B \xrightarrow{1_B} B$, so it is certainly a short exact sequence.  It is easy to see that $\psi\oplus 1_B$ is a $\class{B}$-precover. To check the automorphism requirement for a cover, assume  $f \colon C\Oplus B \xrightarrow{} C\Oplus B$ satisfies $(\psi\oplus 1_B)\circ f = \psi\oplus 1_B$. Note that these maps are  equivalent to matrices 
$$\begin{bmatrix}
    x & y \\
    z & w \\
\end{bmatrix}  \colon C\Oplus B \xrightarrow{f} C\Oplus B \ \ \ \ , \ \ \ \ \ \ \ \ \ \ \ \ 
\psi\oplus 1_B = 
\begin{bmatrix}
    \psi & 0 \\
    0 & 1_B \\
\end{bmatrix}$$
for some maps $x,y,z,w$ having the appropriate (co)domains.
 One easily checks that $f$ must take the form 
$$f = 
\begin{bmatrix}
    x & y \\
    z & w \\
\end{bmatrix} = \begin{bmatrix}
    x & y \\
    0 & 1_B \\
\end{bmatrix}$$ 
for some automorphism $x \colon C \xrightarrow{} C$ satisfying $\psi x=\psi$, and some map $y \colon B \xrightarrow{} C$ satisfying $\psi y=0$. It is also easily verified that $\begin{bmatrix}
    x^{-1} & -x^{-1}y \\
    0 & 1_B \\
\end{bmatrix}$ is a two-sided inverse of $f = \begin{bmatrix}
    x & y \\
    0 & 1_B \\
\end{bmatrix}$.
\end{proof}

The next lemma is a generalization of \cite[Lemma 2]{Goren-FP-inj}.

\begin{lemma}\label{lemma-direc-summand}
Assume   $\class{B}$ is a covering class and closed under direct summands.
Let $E$ be an exact $\Hom_R(\class{B},-)$-acyclic complex 
$$\cdots \xrightarrow{} E_{n+1} \xrightarrow{d_{n+1}} E_n \xrightarrow{d_n} E_{n-1} \xrightarrow{} \cdots$$ 
 as in Defiition~\ref{Defs-relative-G-inj}~(2), so with each $E_n\in\class{B}$. 
Then each cycle $Z_nE$   has a direct sum decomposition,  $Z_nE = L_n \Oplus B_n$, for some $B_n \in \class{B}$ and $L_n \in \rightperp{\class{B}}$. In particular, any weakly Gorenstein $\class{B}$-injective module has such a direct sum decomposition. 
\end{lemma}

\begin{proof}
For short, we set $Z_n := Z_{n}E$. Since $E$ is exact, we have short exact sequences as shown in the middle row of Diagram (\ref{equation-diagram-splittings}) below, and where each differential factors as $d_n = k_{n-1}e_n$. Each cycle $Z_n$ is weakly Gorenstein  $\class{B}$-injective, so applying Lemma~\ref{lemma-weakly-covers} to $Z_{n-1}$ we obtain a short exact sequence $$0 \xrightarrow{} L_n  \xrightarrow{f_n} C_n  \xrightarrow{\varphi_n} Z_{n-1} \xrightarrow{}0$$ where $\varphi_n \colon C_n \xrightarrow{} Z_{n-1}$ is a surjective $\class{B}$-cover with $L_n \in \rightperp{\class{B}}$. 
The (pre)cover property of $\varphi_n$ provides a morphism of short exact sequences, $(\sigma_n, \beta_n, 1)$, as shown in the bottom two rows of the diagram:
\begin{equation}\label{equation-diagram-splittings}
\begin{CD}
     0   @>>>  L_n  @>f_n>>   C_n    @>\varphi_n>>   Z_{n-1}   @>>>    0 \\
    @.        @V\tau_nVV    @V\alpha_nVV      @|   @.\\
     0   @>>> Z_n @>k_n>>  E_n    @>e_n>>   Z_{n-1}    @>>>    0 \\
    @.        @V\sigma_nVV     @V\beta_nVV      @|   @.\\
     0   @>>> L_n @>f_n>>  C_n    @>\varphi_n>>   Z_{n-1}    @>>>    0 \\
\end{CD}
\end{equation}
The morphism of short exact sequences between the two top rows, $(\tau_n, \alpha_n, 1)$, arises from the exactness of the chain complex $\Hom_R(C_n,E)$. Indeed $(d_{n-1})_*(k_{n-1}\varphi_n) = d_{n-1}(k_{n-1}\varphi_n) = (d_{n-1}k_{n-1})\varphi_n = 0\varphi_n =0$. So there exists $\alpha_n \colon C_n \xrightarrow{} E_n$ such that 
\begin{align*}
 (d_n)_*(\alpha_n)   & = k_{n-1}\varphi_n \\
  d_n\alpha_n   & = k_{n-1}\varphi_n \\
  k_{n-1}e_n\alpha_n   & = k_{n-1}\varphi_n. \\
\end{align*}
By left canceling the monic $k_{n-1}$ we get $e_n\alpha_n  = \varphi_n$. Then $\tau_n \colon L_n \xrightarrow{} Z_{n}$ exists by the universal property of the kernel $k_n$.
Now because $\varphi_n$ is a cover, the definition implies that $\beta_n\alpha_n$  must be an automorphism of $C_n$. Letting $r_n$ be an inverse, so that $1_{C_n} = r_n(\beta_n\alpha_n) = (r_n\beta_n)\alpha_n$, we see that $\alpha_n$ is a split monomorphism with retraction $r_n\beta_n$.
Thus we have a direct sum decomposition 
$$E_n = C_n \Oplus \Ker{(r_n\beta_n)}  = C_n \Oplus \Ker{\beta_n}.$$
By the Five Lemma (or Snake Lemma) we also see that $\sigma_n\tau_n$ is an automorphism of $L_n$. Therefore we obtain a similar direct sum decomposition 
$$Z_n = L_n \Oplus \Ker{\sigma_n}.$$
Moreover, the lower left square of the above commutative diagram is necessarily a pullback square with $\beta_n$ and $\sigma_n$ (split) epimorphisms. It follows that $\Ker{\sigma_n} = \Ker{\beta_n}$ is a direct summand of $E_n \in \class{B}$, and so also in $\class{B}$. Therefore, setting $B_n := \Ker{\sigma_n} = \Ker{\beta_n}$, we've shown $$Z_n = L_n \Oplus B_n$$ where $L_n \in \rightperp{\class{B}}$ and $B_n  \in \class{B}$.
\end{proof}

We now show that the direct summand $L_n$ appearing in the direct sum decomposition of Lemma~\ref{lemma-direc-summand} is in fact  Gorenstein $\class{B}$-injective.  This is inspired by and based on~\cite[Proposition~1]{iacob-weakly-Ding-3}. 

Recall from~\cite{holm} that a class of $R$-modules $\class{B}$ is said to be \textbf{\emph{injectively resolving}} if it contains all injective modules and is closed under extensions and cokernels of monomorphisms.

\begin{theorem}\label{theorem-direc-summand}
Assume $\class{B}$ is an injectively resolving  covering class and closed under direct summands. Then each weakly Gorenstein $\class{B}$-injective module, $Z$, is a direct sum of a Gorenstein $\class{B}$-injective module and a module in $\class{B}$. That is, each $Z \in \textnormal{w}\class{GI}_{\class{B}}$  has  a direct sum decomposition $Z = B \Oplus L$, for some $B \in \class{B}$, and $L \in \class{GI}_{\class{B}}$. 
\end{theorem}

\begin{proof}
We may assume that $Z$ is one of the cycles  of the acyclic complex $E$ of the previous Lemma~\ref{lemma-direc-summand}. By that lemma,  each cycle has a direct sum decomposition,  $Z_nE = L_n \Oplus B_n$, for some $B_n \in \class{B}$ and $L_n \in \rightperp{\class{B}}$. 

Now for each $n$, let $\psi_n \colon I_n \xrightarrow{} L_{n-1}$ be a $\class{B}$-cover of $L_{n-1}$.  By Lemma~\ref{lemma-weakly-covers},  it sits in  a short exact 
$$0 \xrightarrow{} \Ker{\psi_n} \xrightarrow{\kappa_n} I_n  \xrightarrow{\psi_n} L_{n-1} \xrightarrow{}0$$
 with $I_n\in\class{B}$ and $\Ker{\psi} \in \rightperp{\class{B}}$. Since  $\Ker{\psi}, \, L_{n-1} \in \rightperp{\class{B}}$ we in fact have  $I_n\in\class{B}\cap \rightperp{\class{B}}$. So by Lemma~\ref{lemma-injectives}, 
 $I_n$ is an injective module. 
 
Next, we will show that $\Ker{\psi_n} \cong L_n$. 
To prove this, we simply appeal to Lemma~\ref{lemma-cover-sums}. It asserts that the short exact sequence 
$$0 \xrightarrow{} \Ker{\psi_n} \xrightarrow{}  I_n \Oplus B_{n-1}  \xrightarrow{\psi_n\oplus 1} L_{n-1}\Oplus B_{n-1}  \xrightarrow{}0$$
is such that $\psi_n\oplus 1$ is a  $\class{B}$-cover of $L_{n-1}\Oplus B_{n-1} \cong Z_{n-1}$.  
 It follows that this short exact sequence is isomorphic to 
$$0 \xrightarrow{} L_n  \xrightarrow{f_n} C_n  \xrightarrow{\varphi_n} Z_{n-1} \xrightarrow{}0,$$
that is, the one previously constructed in the proof of Lemma~\ref{lemma-direc-summand}. In particular we conclude $ \Ker{\psi_n}\cong L_n$.

So now we have shown that there exist  short exact sequences $$0 \xrightarrow{} L_n \xrightarrow{} I_n  \xrightarrow{} L_{n-1} \xrightarrow{}0,$$ 
one for each integer $n$, each of which has  $I_n$ an injective module and with $L_n \in \rightperp{\class{B}}$. By splicing all these short exact sequences together it is clear that the resulting chain complex, $I$, is an exact $\Hom_R(\class{B},-)$-acyclic complex of injectives. This shows each $L_n \in \class{GI}_{\class{B}}$ and completes the proof of the theorem. 
\end{proof}

In the case that the covering class $\class{B}$ is the right half of an hereditary cotorsion pair, $(\class{C}, \class{B})$, then all of the hypotheses of Theorem~\ref{theorem-direc-summand} hold. So we get the following corollary. 

\begin{corollary}\label{corollary-cot}
Assume   $(\class{C}, \class{B})$ is an hereditary cotorsion pair with $\class{B}$ a covering class.
Then each weakly Gorenstein $\class{B}$-injective module is the direct sum of some Gorenstein $\class{B}$-injective module with some module in $\class{B}$.  
\end{corollary}


\subsection{The weakly Gorenstein $\class{B}$-injective model structure}
 We continue with the same assumptions on $\class{B}$ and with the same notation:  $\class{GI}_{\class{B}}$ denotes the   class of all Gorenstein $\class{B}$-injective modules, and   $\textnormal{w}\class{GI}_{\class{B}}$  denotes  the class of all weakly Gorenstein $\class{B}$-injectives.  
We will use Theorem~\ref{theorem-cots-models} to obtain an abelian model structure on $\rmod$ whose class of fibrant objects is generated by $\textnormal{w}\class{GI}_{\class{B}}$. 

\begin{notation}\label{notation-relative-G-inj}
We set $\class{W}_{\class{B}} := \leftperp{\class{GI}_{\class{B}}}$. 
\end{notation}

$\class{W}_{\class{B}} $ will be the class of trivial objects in our model structure.
Note that $\class{W}_{\class{B}}$ is precisely the class of all modules $W$ such that $\Hom_R(W,E)$ remains exact for all exact $\Hom_R(\class{B},-)$-acyclic complexes of injectives $E$. Indeed it follows from the definition that $W \in \class{W}_{\class{B}}$ if and only if $\Ext^1_R(W,Z_{n}E)=0$ for all $n$ and all such $E$, and this is equivalent to $\Hom_R(W,E)$ being exact. In particular, $\class{B} \subseteq \class{W}_{\class{B}}$. 
We note too that the $\Hom_R(W,E)$ characterization easily implies  that $\class{W}_{\class{B}}$ is a thick class; see also~\cite[Lemma 22]{gillespie-iacob-duality-pairs}.

\begin{corollary}\label{cor-modelstruc}
Let $(\class{C},\class{B})$ be a  hereditary cotorsion pair in $\rmod$, cogenerated by some set, and assume $\class{B}$ is closed under coproducts and pure submodules. 
Then the cotorsion pairs $(\class{C},\class{B})$ and $(\class{W}_{\class{B}}, \class{GI}_{\class{B}})$ satisfy the hypotheses of Theorems~\ref{theorem-cots-models}/\ref{theorem-fibrant-objects}, and  moreover  $\leftperp{(\class{B}\cup \class{GI}_{\class{B}})} = \leftperp{\textnormal{w}\class{GI}_{\class{B}}}$. Thus we have a Hovey triple on $R\textnormal{-Mod}$,
$$\mathfrak{M} = (\class{C}, \class{W}_{\class{B}}, \rightperp{(\leftperp{\textnormal{w}\class{GI}_{\class{B}}})}).$$
 This is a cofibrantly generated and hereditary abelian model structure.
The (bi)fibrant objects are characterized as in Theorem~\ref{theorem-fibrant-objects} and Proposition~\ref{prop-bifibrant-acyclicity}. In particular:
\begin{itemize}
\item  $M$ is fibrant, i.e., $M \in  \rightperp{(\leftperp{\textnormal{w}\class{GI}_{\class{B}}})}$,   if and only  if, $M=Z_{0}X$ for some exact  $\Hom_R(\class{C}\cap\class{W}, -)$-acyclic complex $X$ of $R$-modules with each  $X_n \in \class{B}$.
\item  $M$ is bifibrant, i.e.,  $M \in \class{C} \cap \rightperp{(\leftperp{\textnormal{w}\class{GI}_{\class{B}}})}$, if and only if,  $M=Z_{0}X$ for some exact  complex $X$ satisfying that each  $X_n \in \class{C}\cap\class{B}$, and  $\Hom_{R}(N,X)$ is an exact complex  for any $N\in\class{C}\cap\class{W}_{\class{B}}$, and also  $\Hom_{R}(X,N)$ is an exact complex  for any $N\in\class{B}$.
\end{itemize}
\end{corollary}

As we point out in the proof, these conditions on $\class{B}$ imply that it must be a covering class.

\begin{proof}
Note that the hereditary hypothesis automatically implies that $\class{B}$ is also closed under taking quotients by pure submodules and hence under direct limits.  
So then, by a standard argument using purity, there exists a set $\class{T}$ such that each module in $\class{B}$ is  a transfinite extension of modules in $\class{T}$.  It  follows from \cite[Theorem 26]{gillespie-iacob-duality-pairs} that  $(\class{W}_{\class{B}}, \class{GI}_{\class{B}})$ is an injective cotorsion pair,  and it is cogenerated by a set. 
 As previously noted, $\class{B}\subseteq \class{W}_{\class{B}}$, so the  hypotheses of Theorems~\ref{theorem-cots-models}/\ref{theorem-fibrant-objects} are all met. It gives us a  Hovey triple: $\mathfrak{M} = (\class{C}, \class{W}_{\class{B}}, \rightperp{(\leftperp{(\class{B}\cup \class{GI}_{\class{B}})} )} )$.   
The corresponding model structure is cofibrantly generated since each of its cotorsion pairs are cogenerated by a set. Indeed $\class{C}\cap\class{W}_{\class{B}}$ is a deconstructible class by~\cite[Prop.~1.7/Prop.~2.9(1)]{stovicek-deconstructible}.

Next, we note that our hypotheses also guarantee $\class{B}$ is a  covering class, by \cite[Theorem~2.5]{holm-jorgensen-precovers-purity}. So the hypotheses of Corollary~\ref{corollary-cot} are also met. Therefore,  each weakly Gorenstein $\class{B}$-injective module  takes the form  $Z = B \Oplus L$, for some $B \in \class{B}$, and $L \in \class{GI}_{\class{B}}$. Using this we prove what we have  claimed:   
$$\leftperp{(\class{B}\cup \class{GI}_{\class{B}})} = \leftperp{\textnormal{w}\class{GI}_{\class{B}}}.$$
 Indeed, we already know 
 $\leftperp{(\class{B}\cup \class{GI}_{\class{B}})} = \class{C}\cap\class{W}_{\class{B}}$.  We now show that $\leftperp{\textnormal{w}\class{GI}_{\class{B}}} =  \class{C}\cap\class{W}_{\class{B}}$ too. 
Indeed, first note that both $\class{B}\subseteq   \textnormal{w}\class{GI}_{\class{B}}$ and $\class{GI}_{\class{B}} \subseteq  \textnormal{w}\class{GI}_{\class{B}}$. It follows that both 
$\leftperp{\textnormal{w}\class{GI}_{\class{B}}} \subseteq  \leftperp{\class{B}} = \class{C}$ and $\leftperp{\textnormal{w}\class{GI}_{\class{B}}}  \subseteq  \leftperp{\class{GI}_{\class{B}}} = \class{W}_{\class{B}} $.  Hence   $\leftperp{\textnormal{w}\class{GI}_{\class{B}}} \subseteq  \class{C}\cap \class{W}_{\class{B}} $. To show the reverse containment, $\class{C}\cap \class{W}_{\class{B}}  \subseteq  \leftperp{\textnormal{w}\class{GI}_{\class{B}}}$,  let $W \in \class{C}\cap \class{W}_{\class{B}} $ and  $Z \in \textnormal{w}\class{GI}_{\class{B}}$. We are to show $\Ext^1_R(W,Z)=0$. 
But as already noted,  $Z = B \Oplus L$  for some $B \in \class{B}$, and $L \in \class{GI}_{\class{B}}$. So,  
$$\Ext^1_R(W,Z) = \Ext^1_R(W , B) \Oplus \Ext^1_R(W,L). $$
We have $ \Ext^1_R(W,B)=0$ since  $W\in\class{C}$, and $\Ext^1_R(W,L)=0$ since $W\in\class{W}_{\class{B}} $. Therefore, $\Ext^1_R(W,Z)=0$,  and we have shown $\class{C}\cap \class{W}_{\class{B}}   =  \leftperp{\textnormal{w}\class{GI}_{\class{B}}}$.
\end{proof}

\section{The weakly Ding injective modules and model structure}\label{section-weakly-Ding-model}

In  this section we let $R$ be a left coherent ring. Applying  Corollary~\ref{cor-modelstruc}, we next obtain an abelian model structure on $\rmod$ whose  class of  fibrant objects is generated by the class of all weakly Ding injective modules, from~\cite{iacob-weakly-Ding-3}.

Indeed, since $R$ is left coherent,  the class of all finitely presented modules  cogenerates a complete hereditary cotorsion pair $(\class{FP}, \class{FI})$, where  $\class{FI}$ is  the class of all  \emph{$FP$-injective} modules (i.e., \emph{absolutely pure} modules). Recall that the modules in $\class{FP}$ are called \emph{$FP$-projective}, and they are precisely the direct summands of transfinite extensions of finitely presented modules.  Taking $\class{B} = \class{FI}$ in  Definition~\ref{Defs-relative-G-inj},  we obtain the class $\class{DI}$, of all \emph{Ding injective} modules, as well as the class $\textnormal{w}\class{DI}$  of all \emph{weakly Ding injective} modules.  Again, the weakly Ding injective modules were studied in  \cite{Goren-FP-inj}, \cite{iacob-weakly-Ding-1}, \cite{iacob-weakly-Ding-2}, and \cite{iacob-weakly-Ding-3}.
We get the following result.

\begin{theorem}\label{theorem-weakly-Ding-injectives}
Let $R$ be a left coherent ring. 
Then the FP-injective cotorsion pair $(\class{FP},\class{FI})$, along with the Ding injective cotorsion pair $(\class{W}, \class{DI})$, satisfy the hypotheses of Theorems~\ref{theorem-cots-models}/\ref{theorem-fibrant-objects}.  Moreover, $\leftperp{(\class{FI}\cup \class{DI})} = \leftperp{\textnormal{w}\class{DI}}$, where $\textnormal{w}\class{DI}$ is the class of weakly Ding injective modules. We obtain a  cofibrantly generated and hereditary abelian model structure on $R\textnormal{-Mod}$,
$$\mathfrak{M}_{fp} = (\class{FP}, \class{W}, \rightperp{(\leftperp{\textnormal{w}\class{DI}})}).$$
The (bi)fibrant objects are characterized as in Theorem~\ref{theorem-fibrant-objects} and Proposition~\ref{prop-bifibrant-acyclicity}. In particular:
\begin{itemize}
\item  A module $M$ is fibrant, i.e. 
$M \in  \rightperp{(\leftperp{\textnormal{w}\class{DI}})}$,   if and only  if,  $M=Z_{0}X$ for some exact  $\Hom_R(\class{FP}\cap\class{W}, -)$-acyclic complex $X$ of FP-injective modules.
\item  A module $M$ is bifibrant, i.e.,  $M \in \class{FP} \cap \rightperp{(\leftperp{\textnormal{w}\class{DI}})}$, if and only if,  $M=Z_{0}X$ for some exact  complex $X$ satisfying that each  $X_n$ is both FP-projective and FP-injective,  the complex  $\Hom_{R}(N,X)$ is  exact  for any $N\in\class{FP}\cap\class{W}$, and also  $\Hom_{R}(X,N)$ is  exact  for any $N\in\class{FI}$.
\end{itemize}
\end{theorem}

\begin{proof}
This is really an instance of Corollary~\ref{cor-modelstruc}.  It is well known that over a left coherent ring $R$, the class $\class{B}=\class{FI}$ satisfies the hypotheses of Corollary~\ref{cor-modelstruc}; see~\cite{stenstrom-fp}. So the result follows. 
 \end{proof}

 Note that in the above description of the bifibrant modules, $X$ is in particular a totally acyclic complex of FP-pro-injective modules in the following sense. 

\begin{definition}\label{def-FP-pro-inj-mods}
We will call a module in the class $\class{FP}\cap\class{FI}$ an  \textbf{\emph{FP-projective-injective}} module, or for short, an \textbf{\emph{FP-pro-injective}} module.
\begin{enumerate}
\item By a \textbf{\emph{totally acyclic complex of FP-pro-injective modules}} we mean an exact complex $X$ with each $X_n$ an  FP-pro-injective module and with both $\Hom_R(N,X)$ and $\Hom_R(X,N)$ also being exact whenever $N$ is an FP-pro-injective. 
\item By a \textbf{\emph{Gorenstein FP-pro-injective module}} we mean a module $M = Z_0X$ equaling  the 0-cycle module of some 
totally acyclic complex of FP-pro-injective modules, $X$.
\end{enumerate}
We let $\class{GFP}$ denote the full subcategory of all Gorenstein FP-pro-injective modules. 
\end{definition}

\begin{corollary}\label{corollary-weakly-Ding-injectives}
Let $R$ be a left coherent ring such that  each module $M\in \class{FP}\cap\class{W}$ has $FP\textnormal{-id}(M) < \infty$,  i.e.,  finite FP-injective dimension. 
Then the abelian model structure $$\mathfrak{M}_{fp} = (\class{FP}, \class{W}, \rightperp{(\leftperp{\textnormal{w}\class{DI}})})$$ satisfies 
$\class{FP}\cap\rightperp{(\leftperp{\textnormal{w}\class{DI}})} = \class{GFP}$. Thus we may identify the homotopy category as 
 $$\textnormal{Ho}(\mathfrak{M}_{fp}) \simeq \textnormal{St}(\class{GFP}),$$  where $\textnormal{St}(\class{GFP})$ denotes the stable category of  the Frobenius category of all Gorenstein FP-pro-injective modules.
\end{corollary}

\begin{proof}
The two conditions of Theorem~\ref{theorem-stable-CF} are met. Indeed by~\cite[Corollary~4.7]{bazzoni-hrbek-positselski-fp-projective-periodicity}, all of the cycles of any exact complex of  FP-projective modules must   again  be FP-projective. (Here we are using that $R$ is left coherent.) The proof of this can also be found in~\cite[Prop.~3.7]{gillespie-iacob-coherent-regular}.
 \end{proof}


\begin{remark} 
Corollary~\ref{cor-modelstruc} applies nicely to other classes of modules $\class{B}$ that are simultaneously the right half of a  cotorsion pair and the left half of a duality pair in the sense of~\cite{holm-jorgensen-duality}.  The hypotheses of  Corollary~\ref{cor-modelstruc} can be verified in these cases by appealing to~\cite[Theorem~2.5]{holm-jorgensen-precovers-purity}  and~\cite[Theorem~3.1]{holm-jorgensen-duality}.
For example, Theorem~\ref{theorem-weakly-Ding-injectives} generalizes in a straightforward manner to the class of all  \emph{absolutely clean} modules over a general ring $R$. It  is the right half of a complete hereditary cotorsion pair and also part of a  duality pair with the \emph{level} modules. Hence it is a covering class with all of the  necessary properties, and is complemented by the  \emph{(weakly) Gorenstein AC-injective}  modules from~\cite[Section~2]{bravo-gillespie-hovey}.  We leave the details to formulating the generalization of Corollary~\ref{corollary-weakly-Ding-injectives} to the interested reader. In particular, one can appeal to~\cite[Proposition~2.7 and Theorem~5.5]{bravo-gillespie-hovey} for the desired properties. 
\end{remark}


\section{Acyclic complexes of FP-injectives over Ding-Chen rings}\label{Section-Ding-Chen}

Throughout this section we let $R$ be a Ding-Chen ring. 
Recall that this is a two-sided coherent ring  having finite FP-injective dimension when viewed as a (both left and right) module over itself. 
We continue throughout to let $(\class{FP},\class{FI})$ denote the  FP-injective cotorsion pair, and to let $(\class{W}, \class{DI})$ denote the Ding injective cotorsion pair.  
Since $R$ is Ding-Chen, a module is in $\class{W}$ if and only if it has finite flat dimension if and only if it has finite FP-injective dimension; see Section~\ref{subsection-Ding-Chen}.
Therefore, the hypotheses of  Corollary~\ref{corollary-weakly-Ding-injectives} are met. 

We will also use the following notation throughout this section. 

\begin{notation}\label{notation-relative-G-inj}
$\class{Z}$ denotes the class of all cycle  modules  of all  exact complexes   of FP-injective modules. 
\end{notation}

That is, $M\in\class{Z}$ if and only if $M = Z_0X$ for some exact complex $X$ with each $X_n\in\class{FI}$.  

The next result shows that exact complexes of FP-injective modules have properties that are directly analogous to known properties of exact complexes of projective, injective, and flat $R$-modules. As in Definition~\ref{def-FP-pro-inj-mods}, $ \class{GFP}$ denotes the class  of all Gorenstein FP-pro-injective modules.

\begin{proposition}\label{prop-exact-complexes-FP-injectives}
We have  $\class{Z} = \rightperp{(\leftperp{\textnormal{w}\class{DI}})}$,  and   $\class{FP}\cap\class{Z} = \class{GFP}$.   In fact, we have:
\begin{enumerate}
\item   The following are equivalent for a chain complex $X$ of FP-injectives. 
\begin{enumerate}
\item $X$ is exact. 
\item $X$ is exact and $\Hom_R(\class{FP}\cap\class{FI}, - )$-acyclic.  
\item $X$ is exact and $\Hom_R(\class{FP}\cap\class{W}, - )$-acyclic.
\end{enumerate}
\item  The following are equivalent for a chain complex $X$ of FP-pro-injectives. 
\begin{enumerate}
\item $X$ is exact. 
\item $X$ is a totally acyclic complex of FP-pro-injective modules.  
\item $X$ is exact, $\Hom_R(\class{FP}\cap\class{W}, - )$-acyclic, and $\Hom_R(- , \class{FI})$-acyclic.
\end{enumerate}
\end{enumerate}
\end{proposition}

\begin{proof}
Note that if the three statements in (1) are equivalent, then we automatically have $\class{Z} = \rightperp{(\leftperp{\textnormal{w}\class{DI}})}$, by Theorem~\ref{theorem-weakly-Ding-injectives}.    Then $\class{FP}\cap\class{Z} = \class{GFP}$ follows from Corollary~\ref{corollary-weakly-Ding-injectives}.

We prove the equivalence of the statements in (1). It is clear that (c) $\implies$ (b) $\implies$ (a), so we only need to show (a) $\implies$ (c).
 For this, suppose that  $X$ is an exact complex of FP-injective modules. There is known to be a complete cotorsion pair, $(\class{W}_{\textnormal{co}}, \dwclass{I})$, where $\dwclass{I}$ denotes the class of all chain complexes of injective $R$-modules and $\class{W}_{\textnormal{co}}$ is the class of all chain complexes that are \emph{coacyclic} in the sense of Becker~\cite{becker}. As the cotorsion pair has enough projectives, we may  write a short exact sequence 
$0 \xrightarrow{} I \xrightarrow{} W  \xrightarrow{} X \xrightarrow{} 0$ having  $W\in\class{W}_{\textnormal{co}}$ and $I\in \dwclass{I}$.
Since $I$ is a complex of injective modules, this short exact sequence is split in each degree. So applying $\Hom_R(N,-)$ for any given module $N$, we get another short exact sequence of complexes of abelian groups:
$$0 \xrightarrow{} \Hom_R(N, I) \xrightarrow{} \Hom_R(N, W)  \xrightarrow{} \Hom_R(N, X) \xrightarrow{} 0.$$ 
Since $W$ and $X$ are both exact, so is $I$, and therefore  $\Hom_R(N,I)$ is exact whenever $N\in \class{W}$; see \cite[Proposition~7.9]{stovicek-purity}  or \cite[Corollary~4.2]{gillespie-ding-modules}.  Also, since the class of all FP-injective modules is  closed under extensions,  $W$ is a coacylic complex of FP-injective modules. Hence it has FP-injective cycles by  \cite[Proposition~6.11]{stovicek-purity},  (or see~\cite[Corollary~3.6]{gillespie-iacob-coherent-regular}), and so $\Hom_R(N, W)$ is exact whenever $N$ is an FP-projective module. We conclude that both  $\Hom_R(N,I)$ and $\Hom_R(N,W)$, and hence $\Hom_R(N,X)$, are all  exact complexes whenever $N \in \class{FP}\cap\class{W}$. 

The equivalence of the three statements in part (2) now also follows. Indeed as we already noted in the proof of Corollary~\ref{corollary-weakly-Ding-injectives}, the cycles of any exact complex having FP-projective components will automatically have all of its cycles FP-projective. So such a complex is  $\Hom_R(- , \class{FI})$-acyclic.
\end{proof}

By~\cite[Corollary 3.18]{ding and chen 93}, whenever the FP-injective dimensions, $FP\textnormal{-id}(R_R)$ and $FP\textnormal{-id}({}_RR)$, are both finite, then they are equal. If this number is $d <\infty$, then we say that $R$ is a Ding-Chen ring of \emph{dimension} $d$ (or a $d$-FC ring). 

\begin{proposition}\label{prop-FP-dimensions}
Let $R$ be a  Ding-Chen ring of dimension $d$. Then $N\in\class{FP}\cap\class{W}$ if and only if there exists an exact sequence 
$$0 \xrightarrow{} N  \xrightarrow{} W ^0 \xrightarrow{} W^1 \xrightarrow{} \cdots  \xrightarrow{} W^d \xrightarrow{} 0 $$
with each $W^i \in \class{FP}\cap\class{FI}$. 
Moreover, the projective dimension of any $N\in\class{FP}\cap\class{W}$ satisfies  $\textnormal{pd}(N) \leq d$.
\end{proposition}

\begin{proof}
The first statement  is a special case of Lemma~\ref{lemma-CF-dimensions}.  We will show that $\textnormal{pd}(N) \leq d$ whenever $N\in\class{FP}\cap\class{W}$. First, since $N\in\class{W}$,  its  flat dimension satisfies $\textnormal{fd}(N)\leq d$,  and so we may write an exact sequence  $$0 \xrightarrow{} F  \xrightarrow{} P_{d-1} \xrightarrow{} \cdots  \xrightarrow{} P_0 \xrightarrow{} N \xrightarrow{} 0$$
where each $P_i$ is projective and   $F$ is necessarily flat. We are to show that $F$ is even projective. To do so, write another short exact sequence $0 \xrightarrow{} K  \xrightarrow{} P \xrightarrow{} F \xrightarrow{} 0$  with $P$ projective. Then  $K$ too is necessarily flat, and so  $FP\textnormal{-id}(K)\leq d$. Since $N\in\class{FP}$ we conclude $\Ext^{d+1}_R(N,K) = 0$.  By dimension shifting we get $\Ext^{d+1}_R(N,K) \cong \Ext^{1}_R(F,K) = 0$. This implies that the above short exact sequence must split.  In particular, $P \cong  K \Oplus F$, hence  $F$  is projective. 
\end{proof}


Recall from Section~\ref{subsection-Ding-Chen}  that we let  
$\textnormal{Stmod}(R) := \rmod/\class{W}$ denote  the stable module category of $R$. Any abelian model structure on $\rmod$ having $\class{W}$ as its class of trivial objects will recover  $\textnormal{Stmod}(R)$ as its homotopy category. We get the following main result.

\begin{theorem}\label{theorem-exact-FP-injectives}
Let $R$ be a  Ding-Chen ring.
Then   $(\class{FP}\cap\class{W},\class{Z})$ is  a complete hereditary cotorsion pair, cogenerated by a set, and we even have  a cofibrantly generated  abelian model structure on $R\textnormal{-Mod}$,
$$\mathfrak{M}_{fp} = (\class{FP}, \class{W}, \class{Z}).$$
It satisfies $\class{FP}\cap\class{Z} = \class{GFP}$, and consequently
 $$\textnormal{Stmod}(R) \simeq \textnormal{St}(\class{GFP}),$$ 
 where $\textnormal{St}(\class{GFP})$ denotes the stable category of  the Frobenius category of all Gorenstein FP-pro-injective modules. 

Moreover, the following statements are equivalent and further characterize the class  $\class{Z}$ of fibrant objects: 
\begin{enumerate}
\item $M \in \class{Z}  = \rightperp{(\leftperp{\textnormal{w}\class{DI}})} = \rightperp{(\leftperp{(\class{FI}\cup\class{DI})})}$.
\item There is a short exact sequence $0 \xrightarrow{} D \xrightarrow{} A \xrightarrow{} M \xrightarrow{}0$ with $A \in \class{FI}$ and $D\in\class{DI}$.
\item $\Ext^i_R(W,M)=0$ for all FP-pro-injective modules  $W$  and all $i\geq1$.
\item  $M=Z_{0}X$ for some exact  $\Hom(\class{FP}\cap\class{FI}, -)$-acyclic complex $X$ of FP-injective modules.
\item There is a $\Hom_R(\class{FP}\cap \class{FI},-)$-exact short exact sequence $0 \xrightarrow{} M \xrightarrow{} D \xrightarrow{} A \xrightarrow{}0$ with $A \in \class{FI}$ and $D\in\class{DI}$.
\end{enumerate}
\end{theorem}


\begin{proof}
The first part of the theorem is immediate from Corollary~\ref{corollary-weakly-Ding-injectives} and Proposition~\ref{prop-exact-complexes-FP-injectives}. Let us prove the equivalent characterizations of the fibrant objects.

Being special cases  of Theorem~\ref{theorem-cots-models}
(1)(3), we get the equivalence of statements (1) and (2). 
Moreover, these statements certainly imply (3) since  $\class{FP}\cap\class{FI}\subseteq \class{FP}\cap\class{W}$ and the cotorsion pair is hereditary. 
We now  prove that, conversely, (3) implies (1). So assume that $\Ext^i_R(W,M)=0$ for all modules $W\in \class{FP}\cap \class{FI}$ and all $i\geq1$.  Letting $N\in\class{FP}\cap\class{W}$, we are to show $\Ext^1_R(N,M)=0$. By Proposition~\ref{prop-FP-dimensions} there is an exact sequence 
$$0 \xrightarrow{} N  \xrightarrow{} W ^0 \xrightarrow{} W^1 \xrightarrow{} \cdots  \xrightarrow{} W^d \xrightarrow{} 0 $$
with each $W^i \in \class{FP}\cap\class{FI}$. By dimension shifting we have $$\Ext^1_R(N,M) = \Ext^{d+1}_R(W^d,M)=0.$$
 We have now shown  (1)--(3) are equivalent.

Statements (1) and (4) are  equivalent by Proposition~\ref{prop-exact-complexes-FP-injectives}. Finally, statement (5) is also equivalent to the others because it is an instance  of Theorem~\ref{theorem-fibrant-objects}(5).  Note too that the existence of  a  short exact sequence 
as in statement (5)  easily implies  statement (3).  
\end{proof}



One could ask whether or not  $\class{Z}$ contains anything more than  the class $\class{GI}$ of Gorenstein injectives.  So it seems worthwhile to note that if $\class{GI} = \class{Z}$, then $R$ must be left Noetherian. Indeed $\class{GI} = \class{Z}$ implies that each FP-injective module is Gorenstein injective. But over a Ding-Chen ring, every Gorenstein injective module is Ding injective. Hence any FP-injective   must be in $\class{W}\cap\class{DI}$, the core of the Ding injective cotorsion pair, $(\class{W},\class{DI})$.   It follows that every FP-injective is injective and hence $R$ is left Noetherian.


\subsection{Complete duality with Gorenstein flat modules}

We continue to let $R$ denote a Ding-Chen ring.  We now display a perfect  (character module) duality  between the modules  in the class $\class{Z}$  and those  in the class $\class{GF}$ of \emph{Gorenstein flat} modules.  We will let $R^\circ$ denote the opposite ring of $R$. Then a left (resp. right) $R^\circ$-module is equivalent to a right (resp. left) $R$-module. In  this notation, recall that for any given $R$-module $M$, its \emph{character module} is defined to be $M^+ = \Hom_{\Z}(M,\Q)$, and that $M^+$ becomes a $R^\circ$-module in a natural way.


\begin{theorem}\label{theorem-duality}
Let $R$ be a  Ding-Chen ring, $\class{Z}$ the class of all  cycles of exact complexes of  FP-injective $R$-modules, and $\class{GF}$   the class of all Gorenstein flat $R^\circ$-modules. Then $(\class{GF},\class{Z})$ is a complete duality pair,  in the sense of~\cite[Def.~7]{gillespie-iacob-duality-pairs}. In particular, $(\class{GF},\class{Z})$ and $(\class{Z}, \class{GF})$ are both (coproduct and product closed)  duality pairs  in the sense of~\cite[Def.~2.1]{holm-jorgensen-duality}. 
\end{theorem}

\begin{proof}
Since $R$ is coherent, it is well known that $(\class{F}, \class{FI})$ is a complete duality pair (in the sense of~\cite[Def.~7]{gillespie-iacob-duality-pairs}), where $\class{F}$ is the class of flat $R^\circ$-modules and $\class{FI}$ is the class of FP-injective $R$-modules. 
Now by  \cite[Props.~7.2/7.3]{gillespie-ding-modules}, an $R^\circ$-module $L$ is Gorenstein flat if and only if it is a cycle of some acyclic complex $X$ of flat $R^\circ$-modules.  So clearly, $Z \in \class{Z}$ implies $Z^+ \in \class{GF}$, and, $L\in\class{GF}$ implies $L^+\in\class{Z}$.  

We must prove the converses. Before we do so, we let $\class{W}^\circ$ denote the class of trivial $R^\circ$-modules. That is, $R^\circ$-modules having finite flat (equivalently, FP-injective) dimension. Observe that $W \in \class{W}$ (resp. $\class{W}^\circ$) if and only if $W^+ \in \class{W}^\circ$ (resp. $\class{W}$).

So now suppose $Z$ is some $R$-module such that that $Z^+\in \class{GF}$. Using that the Ding injective cotorsion pair, $(\class{W},\class{DI})$, has enough projectives, write 
 a short exact sequence $0 \xrightarrow{} D \xrightarrow{} W \xrightarrow{} Z \xrightarrow{}0$ with $W \in \class{W}$ and $D\in\class{DI}$. Applying $\Hom_{\Z}(-,\Q)$, we get a short exact sequence  of $R^\circ$-modules, $$0 \xrightarrow{} Z^+ \xrightarrow{} W^+ \xrightarrow{} D^+ \xrightarrow{}0.$$
  We see that  $D^+\in\class{GF}$, and also  $Z^+ \in \class{GF}$ by assumption. Hence $W^+\in\class{GF}\cap\class{W}^\circ = \class{F}$. It follows that $W \in \class{FI}$, and from Theorem~\ref{theorem-exact-FP-injectives}(2) the above short exact sequence implies that $Z\in\class{Z}$. 
 
Similarly, suppose $L$ is an $R^\circ$-module such that $L^+\in\class{Z}$. Let $(\class{DP},\class{W}^\circ)$ denote the Ding projective cotorsion pair, so here $\class{DP}$ is the class of all Ding projective $R^\circ$-modules. Using  that this cotorsion pair has  enough injectives, we may write  a short exact sequence $0 \xrightarrow{} L \xrightarrow{} W \xrightarrow{} P \xrightarrow{}0$ with $W \in \class{W}^\circ$ and $P\in\class{DP}$. Applying $\Hom_{\Z}(-,\Q)$, we get a short exact sequence  of $R$-modules, 
$$0 \xrightarrow{} P^+ \xrightarrow{} W^+ \xrightarrow{} L^+ \xrightarrow{}0.$$
 We see that  $P^+\in\class{Z}$, and also  $L^+ \in \class{Z}$ by assumption. Hence $W^+\in\class{W}\cap\class{Z} = \class{FI}$. It follows that  $W \in \class{F}$, and so  from \cite[Theorem~4.11(4)]{saroch-stovicek-G-flat}  the above short exact sequence implies that $L\in\class{GF}$.  (Again, the PGF-modules and Ding projective modules are the same thing when $R$ is coherent.)

Finally, we note that since $R$ is coherent, the classes $\class{F}$ and $\class{FI}$ are each closed under set indexed products and coproducts. Hence the same is true of $\class{GF}$ and $\class{Z}$. They are also both closed under direct summands, and $R$ (considered as an $R^\circ$-module) is in $\class{GF}$. So $(\class{GF},\class{Z})$ is a coproduct and product closed complete duality pair. 
\end{proof}


\end{document}